\documentclass[11pt,a4paper,oneside]{amsart}
\usepackage{mathtools,amsthm,amssymb}

\usepackage[justify]{ragged2e}
\usepackage{microtype,hyphenat,needspace,xspace,xcolor}
\usepackage{csquotes}
\usepackage[left=3.25cm,right=3.25cm,top=3cm,bottom=3cm]{geometry}

\usepackage{faktor}
\definecolor{cluster}{HTML}{A81531}
\usepackage{hyperref}
\usepackage[T1]{fontenc}
\usepackage{lmodern}
\usepackage{eucal}
\usepackage{overpic}
\usepackage{tikz-cd}
\usetikzlibrary{calc}
\usetikzlibrary{decorations.pathmorphing}
\usepackage{todonotes}

\renewcommand\to{\mathchoice{\longrightarrow}{\rightarrow}{\rightarrow}{\rightarrow}}

\tikzset{curve/.style={settings={#1},to path={(\tikztostart)
    .. controls ($(\tikztostart)!\pv{pos}!(\tikztotarget)!\pv{height}!270:(\tikztotarget)$)
    and ($(\tikztostart)!1-\pv{pos}!(\tikztotarget)!\pv{height}!270:(\tikztotarget)$)
    .. (\tikztotarget)\tikztonodes}},
    settings/.code={\tikzset{quiver/.cd,#1}
        \def\pv##1{\pgfkeysvalueof{/tikz/quiver/##1}}},
    quiver/.cd,pos/.initial=0.35,height/.initial=0}

\tikzset{tail reversed/.code={\pgfsetarrowsstart{tikzcd to}}}
\tikzset{2tail/.code={\pgfsetarrowsstart{Implies[reversed]}}}
\tikzset{2tail reversed/.code={\pgfsetarrowsstart{Implies}}}
\tikzset{no body/.style={/tikz/dash pattern=on 0 off 1mm}}

\usepackage{tikz}
\usetikzlibrary{fadings}
\usetikzlibrary{patterns}
\usetikzlibrary{shadows.blur}
\usetikzlibrary{shapes}

\newtheoremstyle{custom}
{}
{}
{}
{0pt}
{}
{}
{ }
{\textbf{\thmname{#1} {\thmnumber{{#2}}}}\thmnote{{ {(#3)}}}\textbf{.}}

\newtheoremstyle{custom*}
{}
{}
{\itshape}
{0pt}
{}
{}
{ }
{\textbf{\thmname{#1}\thmnote{{ {#3}}}}\textbf{.}}

\usepackage{dsfont}
\renewcommand{\mathbb}[1]{\mathds{#1}}
\usepackage{epigraph}

\theoremstyle{plain}
\newtheorem{theorem}{Theorem}[section]
\newtheorem{proposition}[theorem]{Proposition}
\newtheorem{corollary}[theorem]{Corollary}
\newtheorem{fact}[theorem]{Fact}
\newtheorem{lemma}[theorem]{Lemma}
\theoremstyle{definition}
\newtheorem{definition}[theorem]{Definition}
\newtheorem{remark}[theorem]{Remark}

\newtheorem{question}[theorem]{Question}
\newtheorem{example}[theorem]{Example}

\theoremstyle{custom*}
\newtheorem*{theorem*}{Main theorem}

\DeclareMathOperator{\chara}{char}

\newcommand{\emptydefinition}{\emptyset}
\renewcommand{\emptyset}{\varnothing}
\renewcommand{\L}{\mathcal{L}}
\renewcommand{\vec}{\overline}
\newcommand{\Lring}{\L_{\mathrm{ring}}}
\newcommand{\Loag}{\L_{\mathrm{oag}}}

\newcommand{\Lval}{\L_{\mathrm{val}}}
\newcommand{\Lthree}{\L_{{K,\Gamma,k}}}
\newcommand{\N}{\mathbb{N}}
\newcommand{\Gal}[2]{\operatorname{Gal}(#1 | #2)}
\newcommand{\infvK}[1]{{#1}_{\infty}}
\newcommand{\induced}{\overline}
\newcommand{\barv}{\induced{v}}

\newcommand{\condsep}{\(Kv_K\) is separably closed}
\newcommand{\condthens}{\(Kv_K\) is not t-henselian}
\newcommand{\conddivelem}{there is \(L \succeq Kv_K\) such that \(v_LL\) is not divisible}
\newcommand{\conddiv}{\(v_KK\) is not divisible}
\newcommand{\conddef}{\((K,v_K)\) is not defectless}
\newcommand{\conddefelem}{there is \(L \succeq Kv_K\) such that \((L,v_L)\) is not defectless}

\renewcommand{\O}{\mathcal{O}}
\newcommand{\m}{\mathfrak{m}}
\DeclareMathOperator{\res}{res}
\title{Definable henselian valuations in positive residue characteristic}
\author{Margarete Ketelsen}
\address{Institut für Mathematische Logik und Grundlagenforschung, Universität Münster}
\email{margarete.ketelsen@uni-muenster.de}
\author{Simone Ramello}
\address{Institut für Mathematische Logik und Grundlagenforschung, Universität Münster}
\email{simone.ramello@uni-muenster.de}
\author{Piotr Szewczyk}
\address{Institut für Algebra, Technische Universität Dresden}
\email{piotr.szewczyk@tu-dresden.de}
\date{\today}
\begin{document}
\maketitle
\begin{abstract}
	We study the question of $\Lring$-definability of non-trivial henselian valuation rings. Building on previous work of Jahnke and Koenigsmann, we provide a characterization of henselian fields that admit a non-trivial definable henselian valuation. In particular, we treat the cases where the canonical henselian valuation has positive residue characteristic, using techniques from the model theory and algebra of tame fields.
\end{abstract}

\section{Introduction}

\par Many successful investigations around the existence and properties of solutions of polynomial equations over a field rely on enriching the field with a valuation. For example, over the field of $p$-adic numbers $\mathbb Q_p$ the valuation ring is an $\Lring$-definable subset, as first observed in the seminal work of Julia Robinson on Hilbert's Tenth Problem \cite{ROBINSON2014299}. Understanding this phenomenon is a classical topic in the model theory of valued fields, and it has striking applications in the investigations around dividing lines for fields, see e.g. Johnson's spectacular classification of dp-finite fields in \cite{johnson2020dpfinite}. For a more thorough survey on definability of valuations we refer the reader to \cite{fehm-jahnke2017survey}.
\par In this manuscript, we will focus on the problem of characterizing which fields admit a non-trivial definable henselian valuation. There are some clear obstructions: for example, if the field is separably closed, then the answer is always \textit{no}. Once this case is excluded, Jahnke and Koenigsmann isolate in \cite{jahnke-koenigsmann2017defining-coarsenings} necessary and sufficient properties of the canonical henselian valuation $v_K$ of a field $K$ (see Subsection \ref{subs:canonical}) for the existence of a definable non-trivial henselian valuation, under the assumption that $\chara{Kv_K} = 0$.

\begin{theorem}[{\cite[Corollary 6.1]{jahnke-koenigsmann2017defining-coarsenings}}]
	\label{th:cor61}
	Let $K$ be a henselian field that is not separably closed, with $\chara{Kv_K} = 0$. Then $K$ admits a definable non-trivial henselian valuation if and only if \textit{at least one} of the following conditions holds:
	\begin{enumerate}
		\item \condsep,
		\item \condthens,
		\item \conddivelem,
		\item \conddiv.
	\end{enumerate}
\end{theorem}

This characterization (and its extension, which we will state in a moment) falls into the general philosophy of Ax-Kochen/Ershov principles. Typically, (model-theoretic) questions about a valued field $(K,v)$ are answered using its residue field $Kv$ and its value group $vK$. In this case, since the question is about a field $K$ with no specified valuation, one ought to isolate a canonical one among all its henselian valuations; this is usually denoted by $v_K$. Answers to model-theoretic questions about $K$, then, should be given in terms of properties of $Kv_K$ and $v_KK$.

\subsection{Our result} We extend the result of Jahnke and Koenigsmann by removing the assumption on the residue characteristic.

\begin{theorem*}
    \label{main-theorem}
    Let \(K\) be perfect, not separably closed, and henselian. If ${\chara{K} = 0}$ and ${\chara{Kv_K} = p > 0}$, assume that $\O_{v_K}/p$ is semi-perfect. 
    Then \(K\) admits a definable non-trivial henselian valuation if and only if \textit{at least one} of the following conditions hold:
    \begin{enumerate}
    	\item \condsep,
    	\item \condthens,
    	\item \conddivelem,
    	\item \conddiv,
    	\item \conddef,
    	\item \conddefelem.
    \end{enumerate}
\end{theorem*}

\begin{remark}
    Note that in conditions 3 and 6, one automatically has that $v_L$ is non-trivial (since the trivial valuation has divisible value group, and is defectless). Moreover, conditions 5 and 6 are trivial when $\chara{Kv_K} = 0$ (for condition 6, see Corollary \ref{six-trivial}), and thus in that case the theorem reduces to Theorem \ref{th:cor61}.
\end{remark}
	The paper is structured as follows:
\begin{itemize}
	\item Section \ref{section:preliminaries} sets the stage, by fixing the notation and a few elementary facts.
	\item Section \ref{section:divisibility} explores the two roles played by divisibility of the value group: on the one hand, in Proposition \ref{definability}, failure of divisibility in an elementary extension of $Kv_K$ is exploited to ($\emptydefinition$-)define a valuation; on the other, in Theorem \ref{thm:meta}, divisibility is the keystone to finding obstructions to the existence of definable valuations.
	\item Section \ref{section:defect} deploys the technology of independent defect to build definable henselian valuations from certain kinds of Galois defect extensions.
	\item Section \ref{section:characterization} puts everything together, providing the full characterization as advertised.
	\item Section \ref{section:examples} provides a few explicit examples of fields that fit into the main result, and discusses a few open questions.
\end{itemize}

\section{Preliminaries}
\label{section:preliminaries}
\subsection{Notation}
Given a valued field $(K,v)$, we denote by $\O_v$ or \(\O_{(K,v)}\) its valuation ring with maximal ideal $\mathfrak m_v$, $Kv$ its residue field and $vK$ its value group. 
We will write $\infvK{vK}$ for $vK \cup \{\infty\}$, and denote by $\operatorname{res}_v\colon \O_v \to Kv$ the residue map. Given a field \(K\), we denote its separable closure by \(K^{\mathrm{sep}}\). 
A field will be called \emph{henselian} if it admits a non-trivial henselian valuation. 
A ring of characteristic $p>0$ is called \emph{semi-perfect} if $x \mapsto x^p$ is surjective (but not necessarily injective). 
If $K \subseteq L$ is a Galois extension, we will write $\Gal L K$ for the corresponding Galois group. 
We will mostly work with the two languages $\Lring = \{+,\;\cdot\;,0,1\}$, and $\Lval = \Lring \cup \{\O\}$, where $\O$ is a unary predicate. 
If $M$ is a first-order structure, $I$ is an index set and $\mathcal U$ is an ultrafilter on $I$, then we denote by $M^{\mathcal U}$ the corresponding ultrapower. 
Unless otherwise stated, \emph{definable} will mean definable \emph{with parameters}. 
Otherwise, we say $\emptydefinition$-definable.

\subsection{Coarsenings and compositions}
\label{subsec:coarsening-decomposition}
We briefly recall some facts about coarsenings and refinements of valuations, which are explained in greater depth in \cite[Section~2.3]{engler2005valued}. 
For two valuations $v$ and $w$ on $K$, we say that $v$ is \emph{finer} than $w$ or is a \emph{refinement} of $w$ (that $w$ is \emph{coarser} than $v$ or is a \emph{coarsening} of $v$), if $\O_v \subseteq \O_w$. 
We identify $v$ and $w$ if they have the same valuation ring, and use $v$ and $\O_v$ interchangeably. We say \(v\) and \(w\) are \emph{comparable} if \(v\) is a coarsening of \(w\), or \(w\) is a coarsening of \(v\). Coarsenings of a fixed valuation ring $\O_v$ are linearly ordered by inclusion and in one-to-one correspondence with convex subgroups of $vK$. We denote by $v_H$ the coarsening of $v$ associated to a convex subgroup $H$ of $vK$.

Given a valuation \(v\) and a coarsening \(w\) corresponding to the convex subgroup \(H\subseteq vK\), we denote the \emph{induced valuation} \(\barv\colon (Kw)^\times \to H\) on \(Kw\) by \(\barv\), with value group \(H\) and residue field \(Kv\). Given a valuation \(w\) on \(K\) and a valuation \(u\) on its residue field \(Kw\), we define the \emph{composition} \(v=u\circ w\) to be the valuation with valuation ring \(\O_v=\res_w^{-1}(\O_u)\). 
The valuation \(u\) has residue field \((Kw)u=Kv\) and value group \(u(Kw)\), a convex subgroup of \(vK\) with \(vK/u(Kw)\cong wK\).

We can write the valuations as places and obtain the diagram

\begin{center}
	\begin{tikzcd}
		K \arrow[r, "w"] \arrow[rr, "v", bend left] & Kw \arrow[r, "u"] & Kv=(Kw)u.
	\end{tikzcd}
\end{center}

\subsection{The canonical henselian valuation}
\label{subs:canonical}
What follows is classical and can be found in greater detail in \cite[Section~4.4]{engler2005valued}.
For any field $K$, one can arrange henselian valuation rings on \(K\) according to whether their residue field is separably closed or not. Namely, one can define
\[
	H_1(K) = \{\O_v \mid (K,v) \text{ henselian and } (Kv)^{\mathrm{sep}} \neq Kv\}
\]
and
\[
	H_2(K) = \{\O_v \mid (K,v) \text{ henselian and } (Kv)^{\mathrm{sep}}=Kv\}.
\]
As above, we identify $v$ with $\O_v$, so we will often write $v \in H_1(K)$ to mean $\O_v \in H_1(K)$.
\par The set $H_1(K)$ is linearly ordered by coarsening, with $K$ as maximum. If $H_2(K) = \emptyset$, then there is a finest valuation in $H_1(K)$, which we denote by $v_K$. Otherwise, we let $v_K$ be the coarsest valuation in $H_2(K)$. We call $v_K$ the \emph{canonical henselian valuation} on $K$.
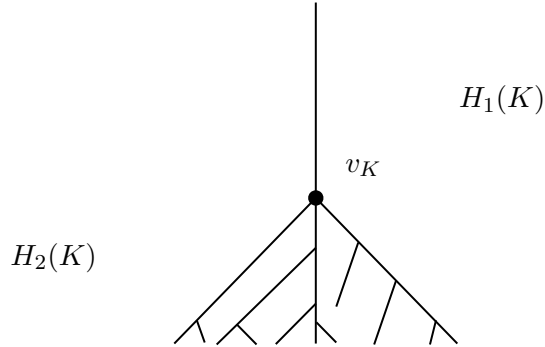
\begin{center}

	\tikzset{every picture/.style={line width=0.75pt}} 

	\begin{tikzpicture}[x=0.75pt,y=0.75pt,yscale=-1,xscale=1]
	
	\draw    (279.6,19.4) -- (279.6,117.4) ;
	\draw    (279.6,117.4) -- (208.6,190.4) ;
	\draw    (279.6,117.4) -- (279.6,190.4) ;
	\draw    (279.6,117.4) -- (350.6,190.4) ;
	\draw [shift={(279.6,117.4)}, rotate = 45.8] [color={rgb, 255:red, 0; green, 0; blue, 0 }  ][fill={rgb, 255:red, 0; green, 0; blue, 0 }  ][line width=0.75]      (0, 0) circle [x radius= 3.35, y radius= 3.35]   ;
	\draw    (279.6,142.4) -- (229.87,190.87) ;
	\draw    (279.6,170.4) -- (259.53,190.67) ;
	\draw    (239.9,180.73) -- (249.87,190.87) ;
	\draw    (279.9,179.73) -- (289.87,189.87) ;
	\draw    (319.9,158.73) -- (308.87,190.87) ;
	\draw    (339.9,178.73) -- (336.87,190.87) ;
	\draw    (300.9,139.73) -- (289.87,171.87) ;
	\draw    (219.9,178.73) -- (223.87,189.87) ;
	
	\draw (350,59.4) node [anchor=north west][inner sep=0.75pt]    {$H_{1}( K)$};
	\draw (125,138.4) node [anchor=north west][inner sep=0.75pt]    {$H_{2}( K)$};
	\draw (293,96.4) node [anchor=north west][inner sep=0.75pt]    {$v_{K}$};

	\end{tikzpicture}
\end{center}
Every henselian valuation on \(K\) is comparable to \(v_K\). Moreover, \(v_K\) is non-trivial if and only if \(K\) is henselian and not separably closed. By definition, \({Kv_K\neq(Kv_K)^{\mathrm{sep}}}\) if and only if \(v_K\in H_1(K)\) if and only if \(H_2(K)=\emptyset\), and if these equivalent conditions hold, then all henselian valuation rings on $K$ are linearly ordered and coarser than $\O_{v_K}$.

\begin{remark}
	\label{lem:correspondence-ref-res}
	Let $(K,v)$ be a henselian valued field with $v \in H_1(K)$. Then, \({v_K=v_{Kv}\circ v}\). 
	Indeed, there is a correspondence between the set $H$ of henselian refinements of $v$ on $K$ and henselian valuations on $Kv$. It is given by sending any refinement $u$ of $v$ to the induced valuation $\induced{u}$ on $Kv$, as in the following diagram (where valuations are written as places):
	\[\begin{tikzcd}
		K && Kv && {(Kv)\induced{u} = Ku}.
		\arrow["{\induced{u}}", from=1-3, to=1-5]
		\arrow["v", from=1-1, to=1-3]
		\arrow["u", curve={height=-30pt}, from=1-1, to=1-5]
	\end{tikzcd}\]
	Since the correspondence preserves residue fields, $H \cap H_1(K)$ is mapped to $H_1(Kv)$, and similarly for $H \cap H_2(K)$.
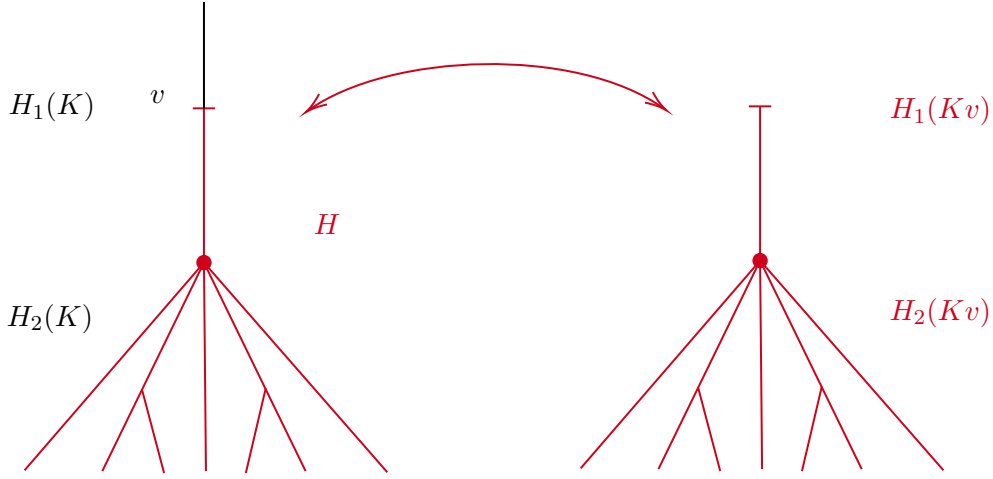
	\begin{center}
		\tikzset{every picture/.style={line width=0.75pt}} 

		\begin{tikzpicture}[x=0.75pt,y=0.75pt,yscale=-1,xscale=1]

		\draw [color={rgb, 255:red, 208; green, 2; blue, 27 }  ,draw opacity=1 ]   (159.42,68.92) -- (159.42,146.25) ;
		\draw [shift={(159.42,146.25)}, rotate = 90] [color={rgb, 255:red, 208; green, 2; blue, 27 }  ,draw opacity=1 ][fill={rgb, 255:red, 208; green, 2; blue, 27 }  ,fill opacity=1 ][line width=0.75]      (0, 0) circle [x radius= 3.35, y radius= 3.35]   ;
		\draw [shift={(159.42,68.92)}, rotate = 270] [color={rgb, 255:red, 208; green, 2; blue, 27 }  ,draw opacity=1 ][line width=0.75]    (0,5.59) -- (0,-5.59)   ;
		\draw [color={rgb, 255:red, 208; green, 2; blue, 27 }  ,draw opacity=1 ]   (159.42,146.25) -- (251.42,251.42) ;
		\draw [color={rgb, 255:red, 208; green, 2; blue, 27 }  ,draw opacity=1 ]   (159.42,146.25) -- (69.42,250.42) ;
		\draw [color={rgb, 255:red, 208; green, 2; blue, 27 }  ,draw opacity=1 ]   (159.42,146.25) -- (108.42,250.83) ;
		\draw [color={rgb, 255:red, 208; green, 2; blue, 27 }  ,draw opacity=1 ]   (159.42,146.25) -- (210.42,250.83) ;
		\draw [color={rgb, 255:red, 208; green, 2; blue, 27 }  ,draw opacity=1 ]   (159.42,146.25) -- (160.42,250.83) ;
		\draw [color={rgb, 255:red, 208; green, 2; blue, 27 }  ,draw opacity=1 ]   (128.42,210.25) -- (139.42,251.83) ;
		\draw [color={rgb, 255:red, 208; green, 2; blue, 27 }  ,draw opacity=1 ]   (190.42,209.25) -- (180.42,251.83) ;
		\draw    (159.42,15.58) -- (159.42,68.92) ;
		\draw [color={rgb, 255:red, 208; green, 2; blue, 27 }  ,draw opacity=1 ]   (438.42,67.92) -- (438.42,145.25) ;
		\draw [shift={(438.42,145.25)}, rotate = 90] [color={rgb, 255:red, 208; green, 2; blue, 27 }  ,draw opacity=1 ][fill={rgb, 255:red, 208; green, 2; blue, 27 }  ,fill opacity=1 ][line width=0.75]      (0, 0) circle [x radius= 3.35, y radius= 3.35]   ;
		\draw [shift={(438.42,67.92)}, rotate = 270] [color={rgb, 255:red, 208; green, 2; blue, 27 }  ,draw opacity=1 ][line width=0.75]    (0,5.59) -- (0,-5.59)   ;
		\draw [color={rgb, 255:red, 208; green, 2; blue, 27 }  ,draw opacity=1 ]   (438.42,145.25) -- (530.42,250.42) ;
		\draw [color={rgb, 255:red, 208; green, 2; blue, 27 }  ,draw opacity=1 ]   (438.42,145.25) -- (348.42,249.42) ;
		\draw [color={rgb, 255:red, 208; green, 2; blue, 27 }  ,draw opacity=1 ]   (438.42,145.25) -- (387.42,249.83) ;
		\draw [color={rgb, 255:red, 208; green, 2; blue, 27 }  ,draw opacity=1 ]   (438.42,145.25) -- (489.42,249.83) ;
		\draw [color={rgb, 255:red, 208; green, 2; blue, 27 }  ,draw opacity=1 ]   (438.42,145.25) -- (439.42,249.83) ;
		\draw [color={rgb, 255:red, 208; green, 2; blue, 27 }  ,draw opacity=1 ]   (407.42,209.25) -- (418.42,250.83) ;
		\draw [color={rgb, 255:red, 208; green, 2; blue, 27 }  ,draw opacity=1 ]   (469.42,208.25) -- (459.42,250.83) ;
		\draw [color={rgb, 255:red, 208; green, 2; blue, 27 }  ,draw opacity=1 ] 
		(212.25,69.58) .. controls (252.89,41.09) and (345.62,37.4) .. (390.08,68.95) ;
		\draw [shift={(391.42,69.92)}, rotate = 216.58] [color={rgb, 255:red, 208; green, 2; blue, 27 }  ,draw opacity=1 ][line width=0.75]    (10.93,-3.29) .. controls (6.95,-1.4) and (3.31,-0.3) .. (0,0) .. controls (3.31,0.3) and (6.95,1.4) .. (10.93,3.29)   ;
		\draw [shift={(210.42,70.92)}, rotate = 323.13] [color={rgb, 255:red, 208; green, 2; blue, 27 }  ,draw opacity=1 ][line width=0.75]    (10.93,-3.29) .. controls (6.95,-1.4) and (3.31,-0.3) .. (0,0) .. controls (3.31,0.3) and (6.95,1.4) .. (10.93,3.29)   ;

		\draw (131,58.4) node [anchor=north west][inner sep=0.75pt]    {$v$};
		\draw (60,59.4) node [anchor=north west][inner sep=0.75pt]    {$H_{1}( K)$};
		\draw (59,165.4) node [anchor=north west][inner sep=0.75pt]    {$H_{2}( K)$};
		\draw (502,61.4) node [anchor=north west][inner sep=0.75pt]  [color={rgb, 255:red, 208; green, 2; blue, 27 }  ,opacity=1 ]  {$H_{1}( Kv)$};
		\draw (502,162.4) node [anchor=north west][inner sep=0.75pt]  [color={rgb, 255:red, 208; green, 2; blue, 27 }  ,opacity=1 ]  {$H_{2}( Kv)$};
		\draw (213,120.4) node [anchor=north west][inner sep=0.75pt]    {$\textcolor[rgb]{0.82,0.01,0.11}{H}$};

		\draw [color={rgb, 255:red, 208; green, 2; blue, 27 }  ,draw opacity=0 ]  (157.42,146.25) -- (161.42,146.25)(159.42,144.25) -- (159.42,148.25) ;
		\draw [color={rgb, 255:red, 208; green, 2; blue, 27 }  ,draw opacity=0 ]  (157.42,146.25) -- (161.42,146.25)(159.42,144.25) -- (159.42,148.25) ;
		\draw [color={rgb, 255:red, 208; green, 2; blue, 27 }  ,draw opacity=0 ]  (157.42,146.25) -- (161.42,146.25)(159.42,144.25) -- (159.42,148.25) ;
		\draw [color={rgb, 255:red, 208; green, 2; blue, 27 }  ,draw opacity=0 ]  (157.42,146.25) -- (161.42,146.25)(159.42,144.25) -- (159.42,148.25) ;
		\draw [color={rgb, 255:red, 208; green, 2; blue, 27 }  ,draw opacity=0 ]  (157.42,146.25) -- (161.42,146.25)(159.42,144.25) -- (159.42,148.25) ;
		\draw [color={rgb, 255:red, 208; green, 2; blue, 27 }  ,draw opacity=0 ]  (436.42,145.25) -- (440.42,145.25)(438.42,143.25) -- (438.42,147.25) ;
		\draw [color={rgb, 255:red, 208; green, 2; blue, 27 }  ,draw opacity=0 ]  (436.42,145.25) -- (440.42,145.25)(438.42,143.25) -- (438.42,147.25) ;
		\draw [color={rgb, 255:red, 208; green, 2; blue, 27 }  ,draw opacity=0 ]  (436.42,145.25) -- (440.42,145.25)(438.42,143.25) -- (438.42,147.25) ;
		\draw [color={rgb, 255:red, 208; green, 2; blue, 27 }  ,draw opacity=0 ]  (436.42,145.25) -- (440.42,145.25)(438.42,143.25) -- (438.42,147.25) ;
		\draw [color={rgb, 255:red, 208; green, 2; blue, 27 }  ,draw opacity=0 ]  (436.42,145.25) -- (440.42,145.25)(438.42,143.25) -- (438.42,147.25) ;
		\end{tikzpicture}

	\end{center}
	The correspondence preserves the coarsening relation. In particular, it follows that $v_K$ is mapped to $v_{Kv}$ and \(v_K=v_{Kv}\circ v\).
\end{remark}
This correspondence immediately allows us to prove that condition 6 of our Main Theorem trivializes in equicharacteristic zero.
\begin{lemma}
	\label{henselian-sep-closed}
	If $Kv_K$ is henselian, then it is separably closed.
\end{lemma}
\begin{proof}
	Suppose $Kv_K$ admits a non-trivial henselian valuation $v$. Via \ref{lem:correspondence-ref-res} this gives rise to a proper refinement of \(v_K\). Hence \(H_2(K)\) is non-empty, so \(Kv_K\) is separably closed.
\end{proof}
\begin{corollary}
	\label{six-trivial}
	Let \(L\succeq Kv_K\). Then \((L,v_L)\) cannot have mixed characteristic.
\end{corollary}
\begin{proof}
	Suppose there is \(L\succeq Kv_K\) such that \((L,v_L)\) has mixed characteristic. 
    Restricting \(v_L\) down to \(Kv_K\) yields a henselian valuation (because \(Kv_K\) is relatively algebraically closed in \(L\)) of mixed characteristic on \(Kv_K\), which in particular must be non-trivial. By Lemma \ref{henselian-sep-closed}, $Kv_K$ must be separably closed.    
	It follows that $L$ is also separably closed, but this implies that $v_L$ is trivial.
\end{proof}

\subsection{Defect and tame fields}\label{sub:defect}
Let $(K,v) \subseteq (L,w)$ be a finite extension of a henselian valued field (note that $w$ is uniquely determined). Let $p$ be the \emph{characteristic exponent} of $Kv$, namely $p = 1$ if $\chara(Kv) = 0$, and $p = \chara(Kv)$ otherwise. Then one has, by \cite[Theorem~3.3.3]{engler2005valued},
	\begin{equation}
	\label{defect-eqn}
		[L:K] = d\cdot(wL:vK)\cdot[Lw:Kv], \tag{\(\star\)}
	\end{equation}
for some $d = p^\nu \in \N$. We say that the extension is \emph{defectless} if $d=1$ (or equivalently $\nu = 0$), otherwise we say that the extension has \emph{defect} (or that it is a \emph{defect extension}). The valued field $(K,v)$ is called \emph{defectless} if all of its finite extensions are defectless. We refer to \cite{kuhlmann2016algebra} for the details on the various equivalent definitions of ``tame''.

\begin{definition}
    Let $(K,v)$ be a henselian valued field and let $p$ be the characteristic exponent of the residue field. Then $(K,v)$ is called \emph{tame} if it is algebraically maximal, $vK$ is $p$-divisible, and $Kv$ is perfect. Equivalently, $(K,v)$ is tame if it is defectless, $vK$ is $p$-divisible, and $Kv$ is perfect.
\end{definition}

In the case of valued fields where $p = 1$ (i.e., valued fields of equicharacteristic zero), tame is equivalent to henselian.

\begin{remark}[{see \cite[Section 7]{kuhlmann2016algebra}}]
    The class of tame valued fields is an elementary class in the language $\Lval$.
\end{remark}

\subsection{t-henselianity and saturation}
A field $K$ is said to be \emph{t-henselian} if it is elementarily equivalent, in $\Lring$, to a field $L$ admitting a non-trivial henselian valuation. However, a field $K$ can be t-henselian without being henselian. The next fact is well-known, and its proof follows from \cite[Lemma~3.3]{prestel-ziegler1978model} and \cite[Page 203]{prestel1991algebraic}.
\begin{fact}
	\label{lem:prestel-ziegler}
	Let $K$ be t-henselian and $\aleph_1$-saturated. Then $K$ admits a non-trivial henselian valuation. In particular, if $K$ is t-henselian and $K \preceq L$ is an $\aleph_1$-saturated elementary extension, then $L$ admits a non-trivial henselian valuation.
\end{fact}

\subsection{Definability in Jahnke-Koenigsmann} For use in later sections, we also summarize two of the main theorems of \cite{jahnke-koenigsmann2017defining-coarsenings}, which will prove to be fundamental tools in our work. In particular, we highlight how the second part of the upcoming statement allows us to often assume that the value group we are working with is divisible.

\begin{definition}
	Let $p$ be a prime. An ordered abelian group $\Gamma$ is \emph{$p$-antiregular} if
	\begin{enumerate}
		\item no non-trivial quotient of $\Gamma$ is $p$-divisible,
		\item $\Gamma$ has no rank $1$ quotient.
	\end{enumerate} 
\end{definition}

\begin{theorem}[{\cite[Theorems A and B]{jahnke-koenigsmann2017defining-coarsenings}}]
	\label{jk}
	Let $K$ be a henselian field which is not separably closed. Assume that $K$ satisfies \textit{at least one} of the following:
	\begin{enumerate}
		\item \condsep,
		\item \condthens,
		\item there is a prime $p$ such that $v_KK$ is not $p$-divisible and not $p$-antiregular,
	\end{enumerate}
	then $K$ admits an $\emptydefinition$-definable non-trivial henselian valuation. In general, if \conddiv, then $K$ admits a definable (possibly with parameters) non-trivial henselian valuation.
\end{theorem}

\section{Divisibility}
\label{section:divisibility}

\subsection{Divisibility as an obstruction} Using classical results on elimination of quantifiers for henselian valued fields of residue characteristic zero, we know that the induced structure of the value group is that of an ordered abelian group. The proof of Theorem \ref{th:cor61} crucially relies on this to obstruct the existence of definable valuations. We use the work of Jahnke and Simon in \cite{jahnke-simon2020nip} to generalize this argument.

\begin{definition}
	Let $\Lthree$ be a three-sorted language with sorts $K$ (with language $\Lring$), $\Gamma$ (with language $\Loag \cup \{\infty\}$) and $k$ (with language $\Lring$), and functions $v\colon K \to \Gamma$ and $\operatorname{res}\colon K \to k$. Any valued field $(K,v)$ can be interpreted as an $\Lthree$-structure in the natural way.
\end{definition}

\begin{fact}[{\cite[Lemma~3.1]{jahnke-simon2020nip}}]
    \label{lem:tame-se}
    Let $(K,v)$ be a tame field of positive residue characteristic. Then the value group $vK$ is purely stably embedded as an ordered abelian group, i.e. for every $D \subseteq vK^n$ which is $\Lthree$-definable (possibly with parameters), there is $\psi \in \Loag(vK)$ which defines $D$.
\end{fact}

The following lemma constitutes the core of the arguments of Jahnke and Koenigsmann, and will be used repeatedly throughout our proofs as well. It relies on a classical fact about divisible ordered abelian groups, which follows from quantifier elimination.
\begin{fact}
    Let $\Gamma$ be a divisible ordered abelian group. Then $\Gamma$ has no proper, $\Loag$-definable (even with parameters) non-trivial convex subgroup.
\end{fact}
\begin{lemma}
	\label{divisibility-tame}
	Suppose $(K,v)$ is a tame valued field of positive residue characteristic, with $vK$ divisible. If $w$ is a proper non-trivial coarsening of $v$, then $w$ cannot be $\Lval$-definable. In particular, it cannot be $\Lring$-definable.
\end{lemma}
\begin{proof}
	Suppose $w$ is $\Lval$-definable (possibly with parameters); 
	then the corresponding proper non-trivial convex subgroup
    \[ 
        \Delta = \{x \in vK \mid \exists y(v(y) = x \land w(y) = 0)\},
    \]
    is $\Lthree$-definable (possibly with parameters) as a subset of $vK$.
	By Fact \ref{lem:tame-se}, $\Delta$ is $\Loag(vK)$-definable. This is a contradiction, as divisible ordered abelian groups have no definable proper non-trivial convex subgroups.
\end{proof}
The proof of the following theorem is then adapted directly from \cite[Corollary 6.1]{jahnke-koenigsmann2017defining-coarsenings}.
\begin{theorem}
    \label{thm:meta}
    Let $(K,v_K)$ be such that $Kv_K$ is perfect of positive characteristic. Assume that $K$ admits a definable non-trivial henselian valuation. Then, \textit{at least one} of the following holds:
    \begin{enumerate}
		\item \condsep,
		\item \condthens,
		\item \conddivelem,
		\item \conddiv,
		\item \conddef,
		\item \conddefelem.
	\end{enumerate}
\end{theorem}
\begin{proof}
    We assume
    \[ \neg 1 \land \neg 2 \land \neg 3 \land \neg 4 \land \neg 5 \land \neg 6,\]
    and want to show that \(K\) then cannot admit a non-trivial definable henselian valuation.
	Note that \(\neg 1\) means that $H_2(K) = \emptyset$. Thus, all henselian valuations on $K$ are linearly ordered and coarser than $v_K$. 
	
	Now, let ${(M,v) \succeq (K,v_K)}$ be an $\aleph_1$-saturated extension in $\Lval$. 
	Then, ${Mv \succeq Kv_K}$ is an $\aleph_1$-saturated elementary extension of a t-henselian field (by $\neg 2$), hence it is henselian by Fact~\ref{lem:prestel-ziegler}. 
    Since ${L := Mv\equiv Kv_K}$ is henselian and not separably closed (by \(\neg 1\)), $v_L$ is non-trivial and in fact, by Remark~\ref{lem:correspondence-ref-res}, we have the following diagram (where the arrows are intended as places),
    \[\begin{tikzcd}
        M && L && {Lv_L = Mv_M}
        \arrow["v", from=1-1, to=1-3]
        \arrow["{v_L}", from=1-3, to=1-5]
        \arrow["{v_M}", curve={height=-28pt}, from=1-1, to=1-5]
    \end{tikzcd}\]
    from which we can deduce that $v_M = v_{L} \circ v$ is a proper refinement of $v$. 
    Now, since $v_LL$ is divisible (because of $\neg 3$), so is $v_MM$. 
    Indeed,
    \[
        vM = {v_MM}/{v_LL},
    \]
    and we know $vM$ to be divisible, as \(v_KK\) is (because of \(\neg 4\)).
    
    Now, by assumption \(Kv_K\) is a perfect field of positive characteristic and thus so is \(L=Mv\succeq Kv_K\). 
    It follows that also the residue field \(Lv_L=Mv_M\) of the valuation \(v_L\) on \(L\) is perfect. 
    By assumptions \(\neg 5\) and \(\neg 6\), \(v\) and $v_L$ are defectless. 
    Thus $v_M$ is defectless (as the composition of defectless valuations is defectless, see \cite[Lemma 2.9]{anscombe2024characterizing}), and  $(M,v_M)$ is a tame valued field. 
    
    Assume now that there is a non-trivial definable henselian valuation $u$ on $K$, with valuation ring defined by the \(\Lring(K)\)-formula $\psi(x)$. 
    Because of \(\neg 1\), $u$ is a (not necessarily proper) coarsening of $v_K$.
    Now, via the elementary embedding, $\psi(x)$ defines a valuation $u^*$ on $M$, which is a coarsening of $v$ and hence a proper coarsening of $v_M$. This is a contradiction to Lemma \ref{divisibility-tame}.
\end{proof}

\subsection{Failure of divisibility as a source of definability} We now seek to extend the arguments given in \cite[Proposition 5.5]{jahnke-koenigsmann2017defining-coarsenings} of the case when $v_KK$ is itself divisible, but there is an elementary extension $L$ of $Kv_K$ which is henselian with non-divisible value group $v_LL$.

\begin{proposition}
	\label{definability}
	Let $K$ be a henselian field, and let $v_K$ be its canonical henselian valuation. If
	\begin{enumerate}
		\item there is \(L\equiv Kv_K\) such that $v_L L$ is non-divisible, \textit{and}
		\item $v_K K$ is divisible,
	\end{enumerate}
	then $K$ admits an $\emptydefinition$-definable non-trivial henselian valuation.
\end{proposition}
\begin{proof} 
	Let \(L\equiv Kv_K\) be such that $v_L L$ is non-divisible.
	First, we notice that $Kv_K$ is not separably closed, and hence by Lemma \ref{henselian-sep-closed} it does not admit any non-trivial henselian valuation.
	Indeed, if $Kv_K$ were separably closed, $L$ would also be separably closed, but separably closed valued fields always have divisible value group, and $v_LL$ is not divisible by assumption.
	
	We now claim that $v_L \in H_1(L)$. 
	Assume for a contradiction that it is not. 
	Then by Theorem~\ref{jk}, there is an $\emptydefinition$-definable non-trivial henselian valuation on $L$. 
	Thus there is also an $\emptydefinition$-definable non-trivial henselian valuation $u$ on $Kv_K$ by elementary equivalence, contradicting the considerations in the previous paragraph.
	
	By the Keisler-Shelah isomorphism theorem \cite[Theorem~2.5.36]{marker2006model}, there are an index set $I$ and an ultrafilter $\mathcal U$ on $I$ such that $L^{\mathcal U} \cong (Kv_K)^{\mathcal U}$. Take the ultrapower $(L^{\mathcal U}, v_L^{\mathcal U}) := (L,v_L)^{\mathcal U}$ as an $\Lval$-structure. Then, $v_L^{\mathcal U}L^{\mathcal U} \equiv v_LL$ is not divisible. Since $v_L \in H_1(L)$, then $v_L^{\mathcal U} \in H_1(L^{\mathcal U})$ and thus $v_L^{\mathcal U}$ coarsens the canonical henselian valuation $v_{L^{\mathcal U}}$ on $L^{\mathcal U}$. In particular, also $v_{L^{\mathcal U}}L^{\mathcal U}$ cannot be divisible, say not $q$-divisible for some prime $q$. Moreover, if we let $(K^{\mathcal U}, v_K^{\mathcal U}) := (K,v_K)^{\mathcal U}$ as an $\Lval$-structure, then $K^{\mathcal U}v_K^{\mathcal U} = L^{\mathcal U}$ and $v_K^{\mathcal U}K^{\mathcal U} \equiv v_KK$ is divisible. Now,
	\begin{itemize}
		\item $K^{\mathcal U}$ is henselian and $v_{K^{\mathcal U}} = v_{L^{\mathcal U}}\circ v_K^{\mathcal U} $ (Remark~\ref{lem:correspondence-ref-res}),
		\item $v_{L^{\mathcal U}}L^{\mathcal U} \subseteq v_{K^{\mathcal U}}K^{\mathcal U}$ is a convex subgroup and thus also \(v_{K^{\mathcal U}}K^{\mathcal U}\) is not $q$-divisible, and
		\item $v_{K^{\mathcal U}}K^{\mathcal U}/v_{L^{\mathcal U}}L^{\mathcal U} \cong v_K^{\mathcal U} K^{\mathcal U}$ is divisible, hence $v_{K^{\mathcal U}}K^{\mathcal U}$ is not $q$-antiregular.
	\end{itemize}
	We can then apply Theorem~\ref{jk} to $(K^{\mathcal U}, v_{K^{\mathcal U}})$ (note the different valuation!) and find an $\emptydefinition$-definable non-trivial henselian valuation on $K^{\mathcal U}$. Since ${K \preceq K^{\mathcal U}}$, the same $\Lring$-formula defines a non-trivial henselian valuation ring on $K$.
\end{proof}

\section{Defect}
\label{section:defect}
Defect -- i.e. the quantity $d$ in equation (\ref{defect-eqn}) -- measures how much of the finite extension of a henselian valued field $(K,v)$ is not induced by extensions of its value group $vK$ and residue field $Kv$. Originally introduced by Ostrowski (see \cite{roquette2002history} for a thorough summary of his work), it was then extensively studied by Franz-Viktor Kuhlmann in his thesis \cite{kuhlmann1990henselian} and his subsequent work. While in classical number theoretic situations, defect is trivial, it is a crucial invariant in the model theory of valued fields, in a certain sense \textit{limiting} the Ax-Kochen/Ershov philosophy. In residue characteristic zero, every extension is defectless; this is why, for example, the classical Ax-Kochen/Ershov theorem can be proven for all henselian equicharacteric zero valued fields.

Valued fields of positive residue characteristic showcase wildly different behaviour. In particular, the fundamental equality can fail, i.e. there might be an extension with $d > 1$. This is usually a serious obstruction to most efforts in understanding complete theories of valued fields. Throughout this section, however, we will deploy such extensions to define henselian valuations.

\subsection{A toy situation} Suppose that $v \in H_1(K)$ and  there exists an $\Lring$-definable subset $D$ of $K$ with
\[
v(D) = \{\gamma \in vK \mid \gamma > H\},
\]
where $H$ is a proper convex subgroup of $vK$.
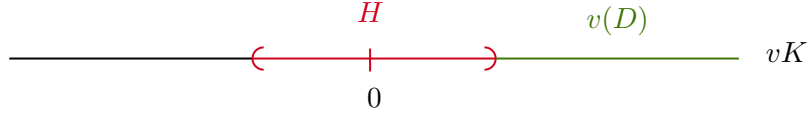
\begin{center}
\tikzset{every picture/.style={line width=0.75pt}} 

\begin{tikzpicture}[x=0.75pt,y=0.75pt,yscale=-1,xscale=1]

\draw    (100,120) -- (222,120) ;
\draw [color={rgb, 255:red, 208; green, 2; blue, 27 }  ,draw opacity=1 ]   (222,120) -- (344,120) ;
\draw [shift={(344,120)}, rotate = 180] [color={rgb, 255:red, 208; green, 2; blue, 27 }  ,draw opacity=1 ][line width=0.75]      (5.59,-5.59) .. controls (2.5,-5.59) and (0,-3.09) .. (0,0) .. controls (0,3.09) and (2.5,5.59) .. (5.59,5.59) ;
\draw [shift={(222,120)}, rotate = 0] [color={rgb, 255:red, 208; green, 2; blue, 27 }  ,draw opacity=1 ][line width=0.75]      (5.59,-5.59) .. controls (2.5,-5.59) and (0,-3.09) .. (0,0) .. controls (0,3.09) and (2.5,5.59) .. (5.59,5.59) ;
\draw [color={rgb, 255:red, 65; green, 117; blue, 5 }  ,draw opacity=1 ]   (344,120) -- (466,120) ;
\draw [color={rgb, 255:red, 208; green, 2; blue, 27 }  ,draw opacity=1 ]   (281,115) -- (281,126.5) ;

\draw (273,90.4) node [anchor=north west][inner sep=0.75pt]  [color={rgb, 255:red, 208; green, 2; blue, 27 }  ,opacity=1 ]  {$H$};
\draw (388,92.4) node [anchor=north west][inner sep=0.75pt]  [color={rgb, 255:red, 65; green, 117; blue, 5 }  ,opacity=1 ]  {$v( D)$};
\draw (478,110.4) node [anchor=north west][inner sep=0.75pt]    {$vK$};
\draw (278,133.4) node [anchor=north west][inner sep=0.75pt]    {$0$};

\end{tikzpicture}
\end{center}
We will prove that whenever $(K,v)$ is as above, then there is a definable coarsening of $v$ which corresponds to $H$, see Subsection \ref{subsec:coarsening-decomposition}. This is a simplified version of the situation that will show up later when working with independent defect extensions. All the main ideas are already contained in this proof, with the advantage that we can avoid the technical subtleties of working modulo an interpretation of a finite field extension, which make the proof of Theorem \ref{lem:ind-val} cumbersome.

\begin{remark}
	\label{rem:beth}
	We will repeatedly use Beth's definability theorem (\cite[Theorem 5.5.4]{hodges1997shorter}). In our case, this means that in order to show that a certain subset $D \subseteq K^n$ is $\Lring(\vec{c})$-definable on $K$, where $\vec{c}$ are some candidate parameters, it is enough to work in the augmented language ${\L = \Lring(\vec{c}) \cup \{P(X_1, \dots X_n)\}}$, and show that that whenever we take two structures ${(L,\vec{c}',D_1), (L,\vec{c}',D_2) \equiv_{\L} (K,\vec{c},D)}$, we have $D_1 = D_2$. In particular, in the proof of Lemma \ref{lem:meta-ind}, we will consider valuation rings and thus unary predicates.
\end{remark}

\begin{lemma}
    \label{lem:meta-ind}
    Let $(K,v)$ be a valued field such that $v \in H_1(K)$. Let $D \subseteq K$ be an $\Lring$-definable subset such that, for some proper convex subgroup $H \subseteq vK$ (possibly equal to the trivial subgroup),
	\[
	\{\alpha \in vK \mid \alpha > H\} = \{v(d) \mid d \in D\}.
	\]
	Let $v_H$ be the coarsening of $v$ corresponding to $H$ and denote by $\O_H$ its valuation ring. Then $v_H$ is definable.
\end{lemma}
\begin{proof}
    We will use Beth's definability theorem as explained in Remark \ref{rem:beth}. We work in $\Lval$: let $L$ and $\O_1, \O_2 \subseteq L$ be such that ${(L,\O_1), (L,\O_2) \equiv (K,\O_H)}$. In particular, both $\O_1$ and $\O_2$ are henselian valuation rings on $L$, corresponding to valuations $v_1$ and $v_2$. We want to show that $\O_1 = \O_2$. 
	Note that, since $v_H \in H_1(K)$, the same holds for $v_1$ and $v_2$, and thus \(\O_1\) and \(\O_2\) are comparable. Without loss of generality, assume $\O_1 \subseteq \O_2$. 
	This implies that $\O_1^\times \subseteq \O_2^\times$.
	We note that, if $D = \psi(K)$ for some $\Lring$-formula $\psi$ (possibly with parameters), then
	\[
		(K,\O_H) \vDash \forall x \left(x \notin \O_H \iff \exists u \exists y \left(u \in \O_H^\times \land \psi(y) \land \left(x = u\frac{1}{y}\right)\right)\right).
	\]
	The same $\Lval$-sentence is then true in both $(L,\O_1)$ and $(L,\O_2)$.

	Now, towards proving $\O_1 = \O_2$, assume for a contradiction that there is $x \in \O_2$ but $x \notin \O_1$: then, there exist $u \in \O_1^\times \subseteq \O_2^\times$ and $d \in \psi(L)$ such that $x = u\frac{1}{d}$. In particular, then, by the $\Lval$-sentence above, $x \notin \O_2$, a contradiction. It follows that $\O_1 = \O_2$.
    \par We can apply Beth's definability theorem and get that $\O_H = \varphi(K)$ for some $\Lring$-formula $\varphi(x)$.
\end{proof}
\begin{remark}
	Beth's definability theorem grants some control over parameters; namely, if $D$ was $\emptydefinition$-definable to begin with, then so is $\O_H$. Moreover, the same argument works if we replace $\Lring$ with some expansion $\L$; we then get that $v_H$ is $\L$-definable.
\end{remark}

\subsection{Independent defect}

Independent defect will be the central tool of our definability proof; the set $\Sigma_L$ which will be defined in a moment, while not $\Lring$-definable per se, will contain enough information to $\Lring$-define a henselian valuation ring modulo the interpretation of a finite field extension.

Note that every equicharacteristic \(0\) valued field is defectless, so for the rest of this section we assume that our valued fields have residue characteristic \(p>0\).

\begin{definition}[{\cite[Introduction]{kuhlmann2023valuation}}]
    Let $(K,v) \subseteq(L,v)$ be a Galois extension of degree $p$ with defect. For any $\sigma \in \Gal{L}{K} \setminus \{\mathrm{id}\}$, we let 
    \[
        \Sigma_\sigma \coloneqq \left\{v\left(\frac{\sigma f - f}{f}\right) \mathrel{\Big|} f \in L^\times\right\}.
    \]
    \end{definition}

\begin{fact}[{\cite[Theorems 3.4 and 3.5]{kuhlmann2023valuation}}]
\label{fact:sigma}
    $\Sigma_\sigma$ does not depend on the choice of $\sigma$, and is a final segment of $\infvK{vK}$. We denote it by $\Sigma_L$.
\end{fact}

\begin{definition}[{\cite[Introduction]{kuhlmann2023valuation}}]
    \label{ind-def}
    The extension $(K,v) \subseteq(L,v)$ has \emph{independent defect} if there is a (possibly trivial) proper convex subgroup $H \subseteq vK$ such that $vK/H$ does not have a minimum positive element, and further
    \[
        \Sigma_L = \{\alpha \in \infvK{vK} \mid \alpha > H\}.
    \]
    We say that the extension has \emph{dependent defect} otherwise.
\end{definition}
    
\begin{definition}[{\cite[Introduction]{kuhlmann2023valuation}}]
         Let \((K,v)\) be a henselian valued field of positive residue characteristic \(p\). If \(\chara(K)=0\), then let \(K'\coloneqq K(\zeta_p)\) where \(\zeta_p\) is a primitive \(p\)th root of unity. Otherwise let \(K'\coloneqq K\). We denote the unique extension of \(v\) to \(K'\) also by \(v\). We say that $(K,v)$ is an \emph{independent defect ﬁeld} if all Galois defect extensions of $(K',v)$ of degree $p$ have independent defect.	
\end{definition}

We now prove a result on composition of independent defect valuations with defectless valuations, which will be quite useful later.

\begin{lemma}
	\label{lem:decomposition-indep-defect}
	Let  $(K,v)$ be a henselian valued field with $v=\barv\circ w$, such that $(K,w)$ is defectless and $(Kw,\barv)$ is a perfect independent defect field. Then $(K,v)$ is an independent defect field.
\end{lemma}
\begin{proof}
    Assume that \(v = \barv \circ w\), with \(w\) defectless and \((Kw, \barv)\) a perfect independent defect field. If \((K,v)\) has mixed characteristic \((0,p)\) we assume without loss of generality that \(K\) contains a primitive \(p\)-th root of unity.
	
	\[\begin{tikzcd}
		K &&& Kw &&& {(Kw)\barv = Kv}
		\arrow["w", "\text{{\tiny defectless}}"', from=1-1, to=1-4]
		\arrow["\barv", "\text{{\tiny independent defect}}"', from=1-4, to=1-7]
		\arrow["{v}", curve={height=-24pt}, from=1-1, to=1-7]
	\end{tikzcd}\]

	Let \((L,v)\) be a Galois defect extension of degree \(p\) of \((K,v)\). 
	In particular, this extension is immediate, i.e., \(Lv=Kv\) and \(vL=vK\).
	We want to show that {\((L,v)\supseteq(K,v)\)} has independent defect.
	
	We first argue that $(Kw,\barv) \subseteq (Lw,\barv)$ is a defect extension. We prove that it is immediate of degree $p$. As \((K,w)\) is defectless, we have the fundamental equality 
	\[p=[L:K]=[Lw:Kw](wL:wK).\]
	We argue that we must have $[Lw:Kw] = p$ and $wL = wK$. Suppose not, i.e. that  \(Lw=Kw\): then \(\barv(Lw)=\barv(Kw)\), and thus
	\[
		wL= \faktor{vL}{\barv(Lw)} = \faktor{vK}{\barv(Kw)} = wK,
	\]
	contradicting \((wL:wK)=p\).

	Using \(wL = wK\) and \(vL=vK\), we get
	\[
	\faktor{vK}{\barv(Kw)} = wK = wL = \faktor{vL}{\barv(Lw)} = \faktor{vK}{\barv(Lw)},
	\]
	and thus \(\barv(Lw)=\barv(Kw)\). Moreover, \((Lw)\barv=Lv=Kv=(Kw)\barv\).
	This shows that \( (Kw,\barv) \subseteq (Lw,\barv) \) is an immediate extension of degree \(p\), in particular a defect extension. Note that as $Kw$ is perfect, the extension is separable; moreover, since $K \subseteq L$ is normal, the same holds for $Kw \subseteq Lw$ (for example, by \cite[Proposition 3.2.16]{engler2005valued}), and thus the latter is Galois.
	Since $(Kw,\barv)$ is an independent defect field, the extension \( (Kw,\barv) \subseteq (Lw,\barv) \) has independent defect. By definition, there is a convex subgroup \(H\) of $\barv(Kw)$ such that $\faktor{\barv(Kw)}{H}$ has no minimum positive element, and further
	\[
	\Sigma_{Lw}\coloneqq\left\{ \barv\left( \frac{f-\tau(f)}{f} \right) \mathrel{\Big|} f\in(Lw)^\times \right\} = \{\gamma\in \infvK{\barv(Kw)} \mid \gamma>H\},
	\]
	where \(\tau\in\Gal{Lw}{Kw}\) is a generator, i.e. \({\langle\tau\rangle=\Gal{Lw}{Kw}}\). Since \(\barv(Kw)\subseteq vK\) is a convex subgroup, \(H\) is also a convex subgroup of the bigger group \(vK\).
	Then, the embedding $\barv(Kw) \subseteq vK$ gives rise to an embedding $\faktor{\barv(Kw)}{H} \subseteq \faktor{vK}{H}$ as convex subgroup, thus $\faktor{vK}{H}$ also has no minimum positive element.
	We will now show that 
	\[
	\Sigma_{L}\coloneqq \left\{ v \left( \frac{f-\sigma(f)}{f} \right) \mathrel{\Big|} f\in L^\times \right\} = \{\gamma \in \infvK{vK} \mid \gamma > H\},
	\]
	where \(\sigma\in\Gal{L}{K}\) is such that \(\langle\sigma\rangle=\Gal{L}{K}\). This will prove that the defect in the extension \((L,v)\supseteq(K,v)\) is independent.

	As $\Sigma_L$ is a final segment in $\infvK{vK}$ (by Fact \ref{fact:sigma}), it is enough to show that \(\Sigma_{L}\cap \infvK{\barv(Kw)} = \Sigma_{Lw}\).
	
	Note that since \((K,w)\) is henselian, \(\O_{(L,w)}\) is fixed setwise by the action of \(\Gal{L}{K}\). Thus by \cite[Proposition~3.2.16(3)]{engler2005valued}, \(\phi\) induces an automorphism \(\induced{\phi}\in\Gal{Lw}{Kw}\) given by 
	\[
	\induced{\phi}(\res_w(f))\coloneqq \res_w(\phi(f)),\quad f\in \O_L.
	\]

	We argue that the map
	\[
        \left\lbrace \begin{array}{rcl}
		\Gal{L}{K}&\to&\Gal{Lw}{Kw}\\
		\phi&\mapsto &\induced{\phi}
	\end{array}\right.
	\]
	is a surjective group homomorphism (cf. the proof of \cite[Lemma~5.2.6]{engler2005valued}): indeed, we write $Lw = Kw(\res_w(a))$, and we take any $\rho \in \Gal{Lw}{Kw}$. We let $a_1, \dots a_p$ be the $\Gal{L}{K}$-conjugates of $a$.
 
    For some $i \leq p$, ${\res_w(a_i) = \rho(\res_w(a))}$, so there is $\sigma \in \Gal{L}{K}$ such that 
    \[
        {\res_w(\sigma(a)) = \induced{\sigma}(\res_w(a)) = \rho(\res_w(a))}.
    \]
    Thus $\induced{\sigma}$ and $\rho$ coincide on $Lw$.
	\par Now, since \((\phi\mapsto\induced{\phi})\) is surjective between finite groups with the same order, it is an isomorphism. It follows that it maps generators of \(\Gal{L}{K}\) to generators of \(\Gal{Lw}{Kw}\).

	We can now compare the sets $\Sigma_L$ and $\Sigma_{Lw}$.

	\textbf{Step 1:} $\Sigma_{Lw}\subseteq\Sigma_{L}\cap \infvK{\barv(Kw)}$.
	
	Fix \(\sigma\in\Gal{L}{K}\) such that \(\langle\sigma\rangle=\Gal{L}{K}\). Then, since the set \(\Sigma_{Lw}\) does not depend on the choice of the generator, we can write
	\begin{align*}
	\Sigma_{Lw}&=\left\{ \barv\left( \frac{f-\induced{\sigma}(f)}{f} \right) \mathrel{\Big|} f\in(Lw)^\times \right\}\\
	&=\left\{ \barv\left( \frac{\res_w(f)-\induced{\sigma}(\res_w(f))}{\res_w(f)} \right) \mathrel{\Big|} f\in \O_{(L,w)}^\times \right\}\\
	&=\left\{ \barv\left(\res_w\left( \frac{f-\sigma(f)}{f} \right)\right) \mathrel{\Big|} f\in \O_{(L,w)}^\times \right\}\\
	&= \left\{ v\left( \frac{f-\sigma(f)}{f} \right) \mathrel{\Big|} f\in \O_{(L,w)}^\times \right\}\cap \infvK{\barv(Kw)} \subseteq\Sigma_{L}\cap \infvK{\barv(Kw)},
	\end{align*}
	since \(\barv(\res_w(x))=v(x)\) for $x\in L$ such that \(v(x)\in \barv(Kw)\).

	\textbf{Step 2:} $\Sigma_{Lw} \supseteq \Sigma_L \cap \infvK{\barv(Kw)}$.
	
	Let \(v\Big(\frac{f-\sigma(f)}{f}\Big)\in\Sigma_L\cap \barv(Kw)\), i.e. \(f \in L^\times\), then because \(wL=wK\), there is \(g\in K^\times\) with \(w(f)=w(g)\) and so for \(h\coloneqq \frac{f}{g}\in \O_{(L,w)}^\times\) we have 
	\[
	\frac{f-\sigma(f)}{f}=\frac{h-\sigma(h)}{h},
	\]
	so
	\[
	v\left(\frac{f-\sigma(f)}{f}\right)=v\left(\frac{h-\sigma(h)}{h}\right)\in\Sigma_{Lw}.
	\qedhere
	\]
\end{proof}

\begin{question}
    Let  $(K,v)$ be a henselian valued field with $v=\barv\circ w$, such that both $(K,w)$ and $(Kw,\barv)$ are independent defect fields. Must then $(K,v)$ be an independent defect field?
\end{question}

\subsection{Defining valuations from independent defect} We now have all the tools needed to build a definable henselian valuation out of an (independent) defect extension. This subsection is structured in several steps:
\begin{itemize}
    \item Theorem \ref{lem:ind-val} contains the core of the argument, showing how to deploy an independent defect extension (and its associated convex subgroup) to build a definable valuation.
    \item Corollary \ref{cor:define-val-from-indep-defect} shows how to go from our assumptions in the Main Theorem to the situation where the right defect extension actually appears.
    \item Corollaries \ref{cor:mix-char} and \ref{cor:compositions-indep-defect-definability} then specialize to the mixed characteristic situation, where some extra care is needed.
\end{itemize}
Finally, in Subsection \ref{subsec:defect-elementary}, we will deal with a scenario that mirrors the one of Proposition \ref{definability}, namely with the situation where $(K,v_K)$ is defectless, but some elementary extension $L$ of $Kv_K$ is henselian and $v_L$ admits defect.

The following theorem should be thought of as a more elevated version of Lemma \ref{lem:meta-ind}, where the $\Lring$-definable set $D$ now appears as a subset of the cartesian product $K^p$, which we identify along an interpretation with a degree $p$ Galois defect extension. For details on interpretations of finite field extensions, see for example \cite[Section 3.5]{chatzidakis2005notes}.

We will often use \cite[Theorem 3.10]{jahnke-koenigsmann2015definable} to reduce to the case where our valuation of interest is in $H_1$. We briefly explain how this is done in the following remark.

\begin{remark}\label{reduction-to-h1}
	Suppose that $K$ is not separably closed, and let $v \in H_2(K)$. Then, \cite[Theorem 3.10]{jahnke-koenigsmann2015definable} produces a non-trivial $\emptydefinition$-definable henselian valuation $w$ on $K$. A closer look at the proof of \cite[Theorem 3.10]{jahnke-koenigsmann2015definable} allows us to argue that $w$ is a coarsening of $v$. If $\chara(K) = p > 0$, or $K$ contains a primitive $p$-th root of unity, then by construction $w$ is a (possibly non-proper) coarsening of $v_K$, and thus in particular of $v$. Otherwise, one moves to $L = K(\zeta_p)$ for some primitive $p$-th root of unity $\zeta_p$. On $L$, one finds a coarsening $w'$ of $v_L$ which is $\emptydefinition$-definable and non-trivial. Its restriction $w \coloneqq w'\vert_K$ is then a definable coarsening of $v_K$ (\cite[Theorem 4.4.3]{engler2005valued} -- with the caveat that $H(K)$ means $H_1(K) \cup \{\O_{v_K}\}$ here).
\end{remark}

\begin{theorem}
    \label{lem:ind-val}
	Let $(K,v)$ be a henselian valued field with $\chara{Kv} = p > 0$ such that \((K,v)\) admits a Galois extension of degree $p$ with independent defect.
	If \(\chara(K)=0\), we additionally assume that \(\zeta_p\in K\), where \(\zeta_p\) is a primitive \(p\)-th root of unity. Then $K$ admits a non-trivial definable henselian valuation, coarsening $v$.
\end{theorem}
\begin{proof}
	Note that we can assume that $v \in H_1(K)$, by Remark \ref{reduction-to-h1}. Let $K \subseteq L$ be the degree $p$ Galois extension with independent defect and take $\theta$ such that $L=K(\theta)$. 
	Let $\sigma$ be a generator of the (cyclic) Galois group $\Gal{L}{K}$, and let
	\[ D \coloneqq \left\{\frac{\sigma(f)-f}{f} \mathrel{\Big|} f \in L^\times\right\}.\]
	As the extension has independent defect, there is a (possibly trivial) convex subgroup $H \subsetneq vK$ such that
	\[ \Sigma_{L} = v(D) = \{\alpha \in \infvK{vK} \mid \alpha > H\}.\]	
	Recall that \(L=K(\theta)\) can be interpreted in \(K\) via the $K$-linear isomorphism 

	\[
f\colon\left\lbrace \begin{array}{rcl}
	K^p&\to&L=K(\theta)\\
	(a_0,\ldots,a_{p-1})&\mapsto &\sum_{i=0}^{p-1} a_i\theta^i
\end{array}\right.
\]

	Note that the coefficients of the minimal polynomial of the generator \(\theta\) over \(K\) are needed as parameters to describe the multiplication. 
	Now, the action of \(\sigma\) is definable in \(K\) (via the interpretation) using the coefficients of the change of basis matrix. If \(\chara(K)=p\) then the matrix has integer coefficients, so no additional parameter is needed. If \(\chara(K) = 0\), then we need to use \(\zeta_p\).
	We will denote the parameter tuple needed to define \(D\) in \(K\) by \(\vec{c}\) (i.e. the coefficients of the minimal polynomial of the generator \(\theta\) over \(K\), and \(\zeta_p\) if necessary).
	
	Thus, \(f^{-1}(D) \subseteq K^p\) is definable in $K$, using $\vec{c}$ as parameters; choose an $\Lring(\vec{c})$-formula $\psi(\vec{X},\vec{c})$ so that $f^{-1}(D) = \psi(K^p,\vec{c})$. Denote by $\O_H \subseteq L$ the coarsening of (the unique extension of) $v$ corresponding to $H$. We can see $\O_H$ as a subset $f^{-1}(\O_H)$ of $K^p$.

	\textbf{Claim:} \textit{$f^{-1}(\O_H) \subseteq K^p$ is $\Lring(\vec{c})$-definable.}

	We will use Beth's definability theorem. We work in ${\L \coloneqq \Lring(\vec{c}) \cup \{P(\vec{X})\}}$, where $P(\vec{X})$ is a $p$-ary predicate. Under the bijection $f$ that establishes the interpretation, the predicate will represent the valuation subring $\O_H$ of $K(\theta)$. We now let $(M,\vec{c}',P_1),(M,\vec{c}',P_2) \equiv_\L (K,\vec{c},f^{-1}(\O_H))$. 
	Note that the interpretation given by \(f\) gives rise to an \(\Lring\)-structure \(N\) and a map \(f'\) which is establishing an interpretation of \(N\) in \(M\).
	
	Furthermore, because $(M,\vec{c}',P_1)$ and $(M,\vec{c}',P_2)$ are elementarily equivalent, we have that: 
	\begin{itemize}
		\item \(M\subseteq N\) are fields, and the polynomial with coefficients from $\vec{c}'$ is of degree \(p\) and irreducible over \(M\);
		thus, the field extension \(M\subseteq N\) is algebraic of degree \(p\);
		\item $\O_1 = f'(P_1)$ and $\O_2 = f'(P_2)$ are non-trivial henselian valuation rings on $N$; we denote the corresponding valuations by \(v_1\) and \(v_2\), respectively;
		\item there is $\ell \geq 2$ such that there is a separable polynomial of degree $\ell$ over $Lv_H$ with no root in $Lv_H$; thus, there are separable polynomials of degree $\ell$ over $Nv_1$ and $Nv_2$ with no root, and thus $v_1, v_2 \in H_1(N)$,
		\item for $i =1,2$, $x \notin \O_i$ if and only if there are $u \in \O_i^\times$ and $y \in \psi(M^p,\vec{c}')$ such that $x = u \frac{1}{f'(y)}$.
	\end{itemize}
	Since $\O_1, \O_2 \in H_1(N)$, they must be comparable.
	Assume, for example, that $\O_1 \subseteq \O_2$. 
	Towards a contradiction, assume now that there is $x \in \O_2$ such that $x \notin \O_1$. 
	In particular, this means that there are $u \in \O_1^\times$ and $y \in \psi(M^p,\vec{c}')$ such that $x = u\frac{1}{f'(y)}$. 
	But since $\O_1^\times \subseteq \O_2^\times$, and thus $u \in \O_2^\times$, we get that $x \notin \O_2$, a contradiction. 
	Then $\O_1 = \O_2$, so $P_1 = P_2$ and we can apply Beth's definability. This means that there is an $\Lring$-formula $\psi(\vec{X},\vec{c})$, such that ${f^{-1}(\O_H) = \psi(K^p,\vec{c})}$. \hfill $\blacksquare_{\text{Claim}}$
	
	Now, we can define the (non-trivial) restriction $\O_H \cap K$ as follows: given $z \in K$, we have that
	\[ z \in \O_H \cap K \iff K \vDash \psi(z, \underbrace{0, \dots, 0}_{(p-1)\text{-times}}, \vec{c}).\]
	This exhibits a non-trivial definable henselian valuation ring on $K$.
\end{proof}

We can now use this theorem to define a non-trivial henselian valuation on certain valued fields with defect, but first we need to produce the Galois defect extensions required to exhibit independent defect.

\begin{lemma}
	\label{magic-defect}
	Let $(K,v)$ be a perfect valued field of residue characteristic $p> 0$. If $(K,v)$ is not defectless, then there is a finite extension $K \subseteq K'$ such that $K'$ admits a defect Galois extension of degree $p$.
\end{lemma}	
\begin{proof}
Take a defect extension $K \subseteq L$. In particular, $p^n \mid [L:K]$ for some $n \geq 1$. We let $N$ be the normal hull of $L$; note that $p^n \mid [N:K]$ and $K \subseteq N$ is Galois. Consider now $H \subseteq \Gal{N}{K}$ to be a $p$-Sylow subgroup. Then $L' \coloneqq N^H$ is such that $L' \subseteq N$ is Galois and it is a tower $L' = L_0 \subseteq L_1 \subseteq \cdots \subseteq L_n = N$ of normal degree $p$ extensions. Moreover, since $K \subseteq L'$ has degree prime to $p$, it is defectless, hence $L' \subseteq N$ is defect. In particular, there is some $\ell \leq n$ such that $L_\ell \subseteq L_{\ell+1}$ is a defect Galois extension of degree $p$, so we take $K' = L_{\ell}$.
\end{proof}	

\begin{remark}
	\label{restriction-interpretation}
	Note that if $F \subseteq E$ is a finite separable field extension of degree $n$, we can interpret $E$ inside of $F$ using the coefficients $\vec{c}$ of the minimal polynomial of a primitive element of the extension. Suppose we had a definable valuation $v$ on $E$, say defined by $\varphi(X,\vec{d})$. Using the interpretation we can define $\O_v$ (as a subset of $F^n$) using some other formula $\widetilde{\varphi}(\vec{X},\vec{d'})$, where now $\vec{d'}$ is given by the coordinates of the elements of $\vec{d}$ in a fixed $F$-basis of $F^n$, together with $\vec{c}$. The restriction $v\vert_F$ can then be defined using the formula 
	\[
		\psi(X,\vec{\rho}) \coloneqq \widetilde{\varphi}(X, \underbrace{0, \dots, 0}_{(n-1)\text{-times}}, \vec{\rho}).
	\]
\end{remark}

\begin{corollary}
	\label{cor:define-val-from-indep-defect}
	Suppose $(K,v)$ is a henselian valued field with $\chara{Kv} = p > 0$ such that \((K,v)\) is not defectless, and such that either
	\begin{enumerate}
		\item $K$ is perfect, if $\chara{K} = p > 0$, or
		\item $\O_v/p$ is semi-perfect, if $\chara{K} = 0$.
	\end{enumerate}
	Then $K$ admits a non-trivial definable henselian valuation, coarsening $v$.
\end{corollary}
\begin{proof}
	We may assume that $v \in H_1(K)$, by Remark \ref{reduction-to-h1}.
	\par In the case where $\chara{K} = 0$, we make the following two observations:
	\begin{itemize}
		\item since $K$ always admits a definable valuation if $v$ is finitely ramified, we may assume that it is infinitely ramified,
		\item the assumptions on $(K,v)$ imply that it is a \emph{roughly deeply ramified} valued field in the sense of \cite[p. 2696]{kuhlmann2023valuation} (note that by \cite[Lemma 4.1]{kuhlmann2023valuation}, we only need to check that $\O_{v}/p$ is semi-perfect, without looking at the completion of $K$). Thus, by \cite[Theorem~1.8]{kuhlmann2023valuation} any finite extension of it is also roughly deeply ramified.
	\end{itemize}
	If $\chara{K} = 0$, let $K_0 = K(\zeta_p)$, where $\zeta_p$ is some primitive $p$th root of unity. Otherwise, take $K_0 = K$. By Lemma \ref{magic-defect}, we have a finite extension $K_0 \subseteq K_1$ which admits a Galois defect extension of degree $p$. Since \(K\subseteq K_1\) is a finite separable field extension, we can find \(b\in K_1\) such that \(K_1=K(b)\).
	
	The field \(K_1\) admits a Galois defect extension of degree $p$, i.e there is $c \in K_1$ which has no Artin-Schreier root in \(K\) if $\chara{K} = p$ (resp. no $p$th root, if $\chara{K} = 0$), and the extension $(K_1,v) \subseteq (K_1(\theta),v)$, where $\theta^p-\theta = c$ (resp. $\theta^p = c$), is an immediate defect extension. Denote by $L = K_1(\theta)=K(b,\theta)$. Then, we have to distinguish two cases: 
	\begin{itemize}
		\item if $\chara{K} = p$, then $K_1$ is a perfect field,
		\item if $\chara{K} = 0$, then $K_1$ is a finite extension of a roughly deeply ramified field, and thus it is roughly deeply ramified itself.
	\end{itemize}
	In both cases, \(K_1\) is an independent defect field (see \cite[Theorem~1.10]{kuhlmann2023valuation}), and thus the extension $(K_1,v) \subseteq (L,v)$ is a degree $p$ defect extension with independent defect. By Theorem~\ref{lem:ind-val}, then, $K_1$ admits a non-trivial definable henselian valuation $v_H$, with valuation ring defined by $\varphi(X,\vec{c})$.  As \(K_1=K(b)\) can be interpreted in \(K\), by Remark \ref{restriction-interpretation} the restriction of $v_H$ to $K$ is then \(\Lring\)-definable with parameters (namely, the coefficients of the minimal polynomial of $b$, and the coordinates of $\vec{c}$ in a chosen $K$-basis of $K_1$).
\end{proof}

As we will later have to deal with definable valuations in some ultrapower of our field $K$, we immediately apply results of Anscombe and Jahnke from \cite{anscombe-jahnke2018henselianity} to get rid of parameters in our definition.
\begin{remark}
	Note that if $(K,v)$ is a henselian valued field of mixed characteristic, with $v \in H_1(K)$, then $v_K$ is also of mixed characteristic.
\end{remark}
\begin{corollary}
	\label{cor:mix-char}
	Suppose $(K,v)$ is a henselian valued field of mixed characteristic, such that \((K,v)\) is not defectless and $\O_v/p$ is semi-perfect. Then $K$ admits an $\emptydefinition$-definable non-trivial henselian valuation.
\end{corollary}
\begin{proof}
	We may assume that $v \in H_1(K)$, by Remark \ref{reduction-to-h1}. Now, $v_K$ is also of mixed characteristic, thus Corollary \ref{cor:define-val-from-indep-defect} yields the existence of a definable non-trivial henselian valuation, and by \cite[Theorem~1.1.(B)]{anscombe-jahnke2018henselianity} we can conclude that there is an \(\emptydefinition\)-definable non-trivial henselian one.
\end{proof}

Later, in the proof of Proposition \ref{prop:defectless->definability-mixed}, we will find ourselves in the situation where we deal with a defect refinement $v_1$ of a valuation $v_K^{\mathcal U}$, where $K^{\mathcal U}v_K^{\mathcal{U}}$ is perfect.

\begin{lemma}   \label{cor:compositions-indep-defect-definability}
	Suppose \((K,v)\) is a henselian valued field of mixed characteristic, such that \((K,v)\) is not defectless, and there is a coarsening \(w\) of \(v\) such that \((K,w)\) is defectless and \(Kw\) is perfect of characteristic $p>0$.
	Then \(K\) admits an \(\emptydefinition\)-definable non-trivial henselian valuation.
\end{lemma}
\begin{proof}
	We may assume that $v \in H_1(K)$: otherwise, we can use \cite[Theorem~3.10]{jahnke-koenigsmann2015definable} to find an $\emptydefinition$-definable non-trivial henselian valuation. Using Lemma \ref{magic-defect}, there is a finite extension \(K(\zeta_p) \subseteq K_1\) such that \((K_1,v)\) admits a defect extension of degree \(p\), and we find \(b\in K_1\) such that \(K_1=K(b)\). 
	
	We now show that \((K_1,v)\) is an independent defect field. For this we consider the decomposition of \((K_1,v)\) with respect to the coarsening \((K_1,w)\) and we denote the induced valuation on \(K_1w\) by \(\barv\):
	\[\begin{tikzcd}
		K_1 && K_1w && {(K_1w)\barv = K_1v}
		\arrow["w", from=1-1, to=1-3]
		\arrow["\barv", from=1-3, to=1-5]
		\arrow["{v}", curve={height=-24pt}, from=1-1, to=1-5]
	\end{tikzcd}\]
	Note that \((K_1,w)\) is defectless, because it is a finite extension of \((K,w)\) and \(K_1w\) is perfect because it is a finite extension of \(Kw\). From Lemma \ref{lem:decomposition-indep-defect} it follows that also \((K_1,v)\) is an independent defect field.
	
	By Theorem~\ref{lem:ind-val}, then, $K_1$ admits a non-trivial definable henselian valuation $v_H$, with valuation ring defined by $\varphi(X,\vec{c})$. As \(K_1=K(b)\) can be interpreted in \(K\), by Remark \ref{restriction-interpretation} the restriction of $v_H$ to $K$ is then \(\Lring\)-definable with parameters (namely, the coefficients of the minimal polynomial of $b$, and the coordinates of $\vec{c}$ in a chosen $K$-basis of $K_1$).

    By our assumptions, $v_K$ is also of mixed characteristic, thus \cite[Theorem~1.1.(B)]{anscombe-jahnke2018henselianity} yields the existence of an \(\emptydefinition\)-definable non-trivial henselian valuation.
\end{proof}

\subsection{Defect in elementary extensions of $Kv_K$}
\label{subsec:defect-elementary}
The following proposition is a direct result of Lemma \ref{cor:compositions-indep-defect-definability}, while we will need some more work for positive characteristic.
\begin{proposition}
	\label{prop:defectless->definability-mixed}
	Let \(K\) be a henselian field such that \((K,v_K)\) has mixed characteristic and \(Kv_K\) is perfect. If
	\begin{enumerate}
		\item there is \(L\equiv Kv_K\) such that \((L,v_L)\) is not defectless, \emph{and}
		\item \((K,v_K)\) is defectless,
	\end{enumerate}
	then \(K\) admits an \(\emptydefinition\)-definable non-trivial henselian valuation.
\end{proposition}
\begin{proof}
	Choose \(L\) such that \(L\equiv Kv_K\) and \((L,v_L)\) is not defectless. By the Keisler-Shelah isomorphism theorem \cite[Theorem~2.5.36]{marker2006model}, take an index set $I$ and an ultrafilter $\mathcal U$ on $I$ such that $(Kv_K)^{\mathcal U} \cong L^{\mathcal U}$. 
	Consider the ultrapowers $(K^{\mathcal U},v_K^{\mathcal U}) \coloneqq (K,v_K)^{\mathcal U}$ and $(L^{\mathcal U},v_L^{\mathcal U}) \coloneqq (L,v_L)^{\mathcal U}$ in \(\Lval\).
	Since $v_L^{\mathcal U}$ is a valuation with defect, while $v_K^{\mathcal U}$ is defectless, the composition $v_1$ is a non-trivial henselian valuation on $K^{\mathcal U}$ with defect. The valuation \(v_1\) is not defectless and of mixed characteristic, since \(v_K^{\mathcal{U}}\) is, and its coarsening \(v_K^{\mathcal{U}}\) is defectless and has perfect residue field \(K^{\mathcal{U}}v_K^{\mathcal{U}}=L^{\mathcal{U}}\).
	Thus by Lemma~\ref{cor:compositions-indep-defect-definability}, there is an \(\emptydefinition\)-definable valuation on \(K^{\mathcal{U}}\).
	By elementary equivalence, we can find an \(\emptydefinition\)-definable henselian valuation on $K$.
\end{proof}

In the case of positive characteristic, eliminating the parameters from the definition of the henselian valuation can be tricky; indeed, a theorem along the lines of \cite[Theorem~1.1.(B)]{anscombe-jahnke2018henselianity} will not be true in general in positive characteristic. However, in our setting (particularly because we can assume $v_KK$ to be divisible) we manage to get an analogous result, where we take a valuation defined in an elementary extension $K \preceq K^*$ (with parameters possibly in $K^*$) and produce a valuation defined with parameters in $K$, thus allowing us to \textquote{push it down}.
\begin{fact}[cf. {\cite[Proposition~2.4]{jahnke2019does}}]
    \label{magic-def}
	Let $(K,v)$ be a henselian valued field with non-separably closed nor real closed residue field. Then $v$ has an $\Lring$-definable refinement.
\end{fact}

\begin{proposition}
	\label{prop:defectless->definability-positive}
	Let \(K\) be a perfect henselian field of positive characteristic. If
	\begin{enumerate}
		\item $v_KK$ is divisible, \emph{and}
		\item there is \(L\equiv Kv_K\) such that \((L,v_L)\) is not defectless, \emph{and}
		\item \((K,v_K)\) is defectless,
	\end{enumerate}
	then \(K\) admits a definable non-trivial henselian valuation.
\end{proposition}
\begin{proof}
		Note that we may assume that $Kv_K$ is not separably closed; otherwise, Theorem \ref{jk} gives an \(\emptydefinition\)-definable non-trivial henselian valuation.
		\par By the Keisler-Shelah isomorphism theorem \cite[Theorem~2.5.36]{marker2006model}, take an index set $I$ and an ultrafilter $\mathcal U$ on $I$ such that ${(Kv_K)^{\mathcal U} \cong L^{\mathcal U}}$.
        Consider the ultrapowers ${(K^{\mathcal U},v_K^{\mathcal U}) \coloneqq (K,v_K)^{\mathcal U}}$ and ${(L^{\mathcal U},v_L^{\mathcal U}) \coloneqq (L,v_L)^{\mathcal U}}$ in \(\Lval\).
        Since $v_L^{\mathcal U}$ is a valuation with defect, the composition $v_1=v_L^{\mathcal U}\circ v_K^{\mathcal U}$ is a non-trivial henselian valuation on $K^{\mathcal U}$ with defect. 
        Thus, by Corollary~\ref{cor:define-val-from-indep-defect}, there is an $\Lring(K^{\mathcal U})$-definable non-trivial henselian valuation $v_1'$ coarsening $v_1$, say given by $\varphi(K^{\mathcal U},\vec{\beta})$ for some parameters $\vec{\beta} \in (K^{\mathcal U})^\ell$, $\ell \geq 1$.
        
        Since $Kv_K$ is not separably closed, Fact~\ref{magic-def} yields an $\Lring(K)$-definable valuation $w$ on $K$ (not necessarily henselian!) such that $\O_w \subseteq \O_{v_K}$. 
        Let $\psi(X,\vec{c})$ be the formula defining $w$.
        Denote by $w^*$ the valuation corresponding to $\psi(K^{\mathcal U},\vec{c})$. 
        Note that since $(K,v_K) \preceq (K^{\mathcal U}, v_K^{\mathcal U})$, then we still have that $w^*$ is a proper refinement of $v_K^{\mathcal U}$.
        
		As in the proof of \cite[Theorem~2.3.4]{engler2005valued}, given two valuation rings $\O_1$ and $\O_2$, we let
		\[
			\O_1 \cdot \O_2 \coloneqq \left\{\frac{x}{y} \mathrel{\Big|} x \in \O_1, \; y \in \O_1 \setminus \m_2 \right\}.
		\]
		Note that $\O_1 \cdot \O_2$ contains both $\O_1$ and $\O_2$ and is always a valuation ring. Indeed, it is the finest valuation ring containing both: if $\O_1, \O_2 \subseteq \O_3$, then for any $x \in \O_1$ and $y \in \O_1 \setminus \m_2$, one has that $y \notin \m_3$, so $\frac{1}{y} \in \O_3$ and thus $\frac{x}{y} \in \O_3$. We now consider the $\Lring(\vec{c})$-definable set 
		\[
			X = \left\{\vec{b} \in (K^{\mathcal U})^\ell \mid \varphi(K^{\mathcal U},\vec{b}) \; \text{is a valuation ring and} \; \O_{w^*}\cdot\varphi(K^{\mathcal U},\vec{b}) \neq K^{\mathcal{U}}\right\}.
		\]
		For each $\vec{b} \in X$, we consider the $\Lring(\vec{c},\vec{b})$-definable valuation ring
        \[
            \O_{\vec{b}} \coloneqq \O_{w^*} \cdot \varphi(K^{\mathcal U},\vec{b}),
        \]
        with corresponding valuation $v_{\vec{b}}$. As all coarsenings of a given valuation are comparable, any $v_{\vec{b}}$ for $\vec{b} \in X$ is comparable with $\O_{v_K^{\mathcal{U}}}$.
        \par Since $v_K^{\mathcal U}$ is a tame valuation with divisible value group, Lemma \ref{divisibility-tame} implies that $v_{\vec{b}}$ cannot be a proper coarsening. In particular, then, $\O_{v_{\vec{b}}} \subseteq \O_{v_K^{\mathcal U}}$. It follows that the union 
		\[
			\O \coloneqq \bigcup_{\vec{b} \in X} \O_{\vec{b}}
		\]
		is a non-trivial $\Lring(\vec{c})$-definable (note the parameters!) valuation ring on $K^{\mathcal U}$.
		\par To show that $\O$ is henselian, we will show that $\O$ contains $\varphi(K^{\mathcal{U}},\vec{\beta})$, i.e. that $\vec{\beta} \in X$. It is enough to show that ${\O_{\vec{\beta}} \neq K^{\mathcal U}}$. Since $\varphi(K^{\mathcal U},\vec{\beta})$ is comparable with $v_K^{\mathcal U}$, we have two possible cases. Either $\varphi(K^{\mathcal U},\vec{\beta})$ is coarser than $v_K^{\mathcal U}$, in which case $\O_{\vec{\beta}} = \varphi(K^{\mathcal{U}},\vec{\beta})$, or $\varphi(K^{\mathcal U},\vec{\beta})$ is finer than $v_K^{\mathcal U}$, in which case $\O_{\vec{\beta}}$ is also a refinement of $v_K^{\mathcal{U}}$. Either way, $\O_{\vec{\beta}}$ is non-trivial.
		\par We have just shown that $\O$ is a coarsening of $\varphi(K^{\mathcal{U}},\vec{\beta})$, thus it is henselian. In particular, the same formula defines a non-trivial henselian valuation on $K$, as needed.
\end{proof}

\section{Proof of the main theorem}
\label{section:characterization}

\begin{theorem*}
    Let \(K\) be perfect, not separably closed, and henselian. If ${\chara{K} = 0}$ and ${\chara{Kv_K} = p > 0}$, assume that $\O_{v_K}/p$ is semi-perfect. 
    Then \(K\) admits a definable non-trivial henselian valuation if and only if \textit{at least one} of the following conditions hold:
    \begin{enumerate}
    	\item \condsep,
    	\item \condthens,
    	\item \conddivelem,
    	\item \conddiv,
    	\item \conddef,
    	\item \conddefelem.
    \end{enumerate}
\end{theorem*}
\begin{proof}    
    If \(\chara{Kv_K} = 0\), the statement is exactly Theorem \ref{th:cor61}. 
    This is because condition $5$ is trivial (as equicharacteristic zero henselian valued fields are defectless), and Corollary~\ref{six-trivial} shows that condition $6$ is also trivial in this case.
    
    Assume \(\chara{K}v_K = p>0\) and note that \(Kv_K\) is perfect: in mixed characteristic, \(Kv_K\) is perfect since \(\O_{v_K}/p\) is semi-perfect; in positive characteristic, \(Kv_K\) is perfect since \(K\) is. Thus, one direction is Theorem~\ref{thm:meta}.
    \par In the other direction, we do a case distinction.
    First, assume $1 \vee 2 \vee 4$, then it follows from Theorem~\ref{jk} that \(K\) admits a definable non-trivial henselian valuation.
    
    If $3 \wedge \neg 4$ then Proposition~\ref{definability} yields a existance of definable non-trivial henselian valuation.

    Suppose now that \(5\) holds, i.e. $(K,v_K)$ is not defectless. 
    Then we get a non-trivial definable henselian valuation from Corollary~\ref{cor:define-val-from-indep-defect}.
    \par Finally, assume that $6$ holds, and since we have already established the case where $1 \vee 2 \vee 3 \vee 4 \vee 5$ holds, we can also assume $\neg 1 \land \neg 2 \land \neg 3 \land \neg 4 \land \neg 5$. 
    Then by Proposition~\ref{prop:defectless->definability-mixed} (in mixed characteristic) and Proposition~\ref{prop:defectless->definability-positive} (in positive characteristic), there is a definable non-trivial henselian valuation on \(K\).
\end{proof}

\section{Examples and questions}
We isolate a few interesting examples where our result yields the existence of non-trivial definable henselian valuations. We then discuss the difficulties in building ones satisfying certain properties. We refer to the six conditions in the Main Theorem as conditions $1$, $2$, $3$, $4$, $5$, and $6$, and to their negations as conditions $\neg 1$, $\neg 2$, $\neg 3$, $\neg 4$, $\neg 5$, and $\neg 6$.
\label{section:examples}

We start with the following observation.

\begin{remark}
	\label{rem:non-hens-sep-closed-canonical}
	Let \((K,v)\) be a henselian valued field with \(Kv\) non-henselian and not separably closed. Then \(v=v_K\) is the canonical henselian valuation on \(K\).
	Indeed, there can be no proper henselian refinements of \(v\), as they would correspond to non-trivial henselian valuations on \(Kv\) via \ref{lem:correspondence-ref-res}. Thus \(H_2(K)\) is empty and \(v\) is the finest valuation in \(H_1(K)\). 
\end{remark}

The aim of this section is to give some examples for fields where our Main Theorem yields a definable non-trivial henselian valuation because of conditions $5$ or $6$ (while at the same time \(\neg1\wedge\neg2\wedge\neg3\wedge\neg4\) holds). 
This boils down to finding a non-henselian t-henselian field that is not separably closed 
and such that we can control some properties of the canonical henselian valuation in some elementary extension 
(to access conditions $3$ and $6$).

The construction of non-henselian t-henselian fields goes back to Prestel and Ziegler,
see \cite[p. 338]{prestel-ziegler1978model}.
Other instances can be found 
in \cite[Proposition~6.7]{fehm-jahnke2015quantifier} (note that the field constructed there is elementary equivalent to a field admitting a non-trivial henselian valuation with divisible value group)
and 
in \cite[Examples~3.8 and 5.4]{jahnke-koenigsmann2017defining-coarsenings} (with an elementary extension \(L\) such that \(v_LL\) is not divisible).
All of these provide fields of characteristic \(0\).
In \cite[Proposition~4.13]{anscombe-jahnke2018henselianity}, there is a construction of a non-henselian t-henselian field of positive characteristic that is elementary equivalent to a field admitting a non-trivial tame valuation with divisible value group. Such a field will serve as the residue field of the canonical henselian valuation to construct examples that satisfy condition $5$, see Example \ref{five}. We will also adapt their construction in Lemma~\ref{building-block} and Proposition~\ref{construction-of-t-defect} to serve as the residue field of the canonical henselian valuation of an example that satisfies condition $6$.

The following is needed to check that the constructed examples satisfy condition \(\neg 3\).

\begin{proposition}
	\label{prop:t-hens:one-divisible-all-divisible}
	Let \(K\) be non-henselian and t-henselian such that there is \(K^*\equiv K\) admitting a non-trivial henselian valuation \(v^*\) with divisible value group. 
	Then for all \(L\equiv K\), \(v_LL\) is divisible.
\end{proposition}
\begin{proof}
	We may assume that \(L\) is henselian and not separably closed. 
	Otherwise \(v_L\) is trivial and \(v_LL\) is the trivial group which is divisible.
	In particular, \(K\) is not separably closed.
	
	By the Keisler-Shelah isomorphism theorem \cite[Theorem~2.5.36]{marker2006model}, there are an index set \(I\) and an ultrafilter \(\mathcal U\) on \(I\) such that \(F\coloneqq (K^*)^{\mathcal U}\cong L^{\mathcal{U}}\). 
	Let \((K^*,v^*)^{\mathcal U}\eqqcolon (F,(v^*)^{\mathcal U})\) and \((L,v_L)^{\mathcal U}\eqqcolon (F,(v_L)^{\mathcal U})\).
	
	We claim that \((v^*)^{\mathcal U}\) and \((v_L)^{\mathcal U}\) are comparable. 
	Suppose not: then, \(Lv_L\) is separably closed, and by Theorem~\ref{jk}(1) \(L\) admits an  \(\emptydefinition\)-definable non-trivial henselian valuation. The same formula then defines an \(\emptydefinition\)-definable non-trivial henselian valuation on \(K\), contradicting the fact that \(K\) is non-henselian.
	
	Now we have to treat two cases. First, if \((v_L)^{\mathcal U}\) is a coarsening of \((v^*)^{\mathcal U}\). Then \((v_L)^{\mathcal U}F\) is a quotient of \((v^*)^{\mathcal U}F\) modulo some convex subgroup. Since \(v^*K^*\) is divisible by assumption, so is \((v^*)^{\mathcal U}F=(v^*K^*)^{\mathcal U}\), which then implies that also \((v_L)^{\mathcal{U}}F=(v_LL)^{\mathcal U}\) and \(v_LL\) are divisible. 
	
	Second, if \((v_L)^{\mathcal U}\) is a proper refinement of \((v^*)^{\mathcal U}\), then \((v^*)^{\mathcal U}F\) is a non-trivial quotient of \((v_L)^{\mathcal U}F\). 
	We assume for a contradiction that \((v_L)^{\mathcal U}F\) is not \(p\)-divisible for some prime \(p\). 
	Now, \((v_L)^{\mathcal U}F\) is also not \(p\)-antiregular, since \((v^*)^{\mathcal U}F\) is a non-trivial quotient, which is divisible.
	Thus, by Theorem~\ref{jk}(3), \(L\) and thus also \(K\) admit \(\emptydefinition\)-definable non-trivial henselian valuations, but \(K\) is non-henselian, a contradiction.
	Hence \((v_L)^{\mathcal U}F=(v_LL)^{\mathcal U}\) is divisible and so is \(v_LL\).
\end{proof}

\begin{remark}
	In the following, we will construct examples where $K$ has positive characteristic. It would be interesting to also produce examples in mixed characteristic. This essentially would amount to finding a way to construct a valued field $(K,v)$ of mixed characteristic, such that $vK$ is divisible, $\O_{v}/p$ is semi-perfect, and $Kv$ is some prescribed residue field. We are not aware of such a construction.
\end{remark}

\subsection{Puiseux series} We use the Puiseux series construction to build an example of a field satisfying conditions $\neg 1 \land \neg 2 \land \neg 3\land \neg 4 \land 5$ of the Main Theorem.

\begin{definition}[{\cite[Definition 4.4]{anscombe-jahnke2018henselianity}}]
	A field $K$ is of \emph{divisible-tame type} if there exist $L \equiv K$ and a non-trivial valuation $w$ on $L$ such that $(L,w)$ is tame with $wL$ divisible.
\end{definition}

\begin{remark}
	In particular, every divisible-tame type field is perfect.
\end{remark}

\begin{fact}[{\cite[Proposition~4.13]{anscombe-jahnke2018henselianity}}]\label{t-hensel-not-hensel}
	Let $p$ be a prime or zero. There exists a non-henselian t-henselian field of characteristic $p$ which is
	\begin{enumerate}
		\item not separably closed, and
		\item of divisible-tame type.
	\end{enumerate}
\end{fact}

We will use the following folklore fact, whose proof we sketch out as we could not find a reference in the literature.

\begin{lemma}
	\label{lem:Puiseux-defect}
	Let \(K_0\) be a field of characteristic \(p>0\). Consider the Puiseux series over $K_0$,
	\[
	K \coloneqq \bigcup_{n \geq 0} K_0(\!(t^{\frac{1}{n}})\!)
	\]
	together with the restriction $v_t$ of the $t$-adic valuation from $K_0(\!(\mathbb{Q})\!)$.
	Then, $(K,v_t)$ is a henselian valued field of positive characteristic, $v_t$ is not defectless, and if $K_0$ is perfect, then so is $K$.
\end{lemma}
\begin{proof}
	As $(K,v_t)$ is the increasing union of henselian fields, $v_t$ is henselian. However, it is not algebraically maximal. Indeed, the equation $X^p-X-\frac{1}{t} = 0$ has no solution in $K$, however it admits a solution in $K_0(\!(\mathbb Q)\!)$, namely $a = \sum_{n \geq 0} t^{-\frac{1}{p^n}}$. Then, $K \subsetneq K(a) \subseteq K_0(\!(\mathbb{Q})\!)$ is a tower of immediate extensions, thus $K \subsetneq K(a)$ is a proper algebraic immediate extension. Note that if $K_0$ is perfect, then $K_0(\!(t^{\frac{1}{n}})\!)^{\frac{1}{p}} = K_0(\!(t^{\frac{1}{np}})\!)$ for every $n$, and thus the union $K$ is perfect.
\end{proof}

\begin{example}\label{five}
	Let $K_0$ be a non-henselian t-henselian field of characteristic $p>0$ which is not separably closed, and such that there exist $L \equiv K_0$ and a non-trivial henselian valuation $w$ on $L$ such that $wL$ is divisible (such a $K_0$ exists, for example, by Fact \ref{t-hensel-not-hensel}). Then,
	\[
	K \coloneqq \bigcup_{n \geq 0} K_0(\!(t^{\frac{1}{n}})\!)
	\]
	satisfies $\neg 1 \land \neg 2 \land \neg 3\land \neg 4 \land 5$ from our Main Theorem. 
	Indeed, we have \(v_K=v_t\) by Remark \ref{rem:non-hens-sep-closed-canonical}.
	Now conditions \(\neg 1\), \(\neg 2\) and \(\neg 4\) are immediate from the construction, condition \(\neg 3\) follows from Proposition \ref{prop:t-hens:one-divisible-all-divisible} 
	and condition \(5\) holds because of Lemma \ref{lem:Puiseux-defect}. 
	It then follows that \(K\) admits a definable non-trivial henselian valuation.
\end{example}

\subsection{Condition $6$ is necessary} We adapt a construction from \cite{anscombe-jahnke2018henselianity} to produce non-henselian, t-henselian fields with defect in some elementary extension. We first introduce a series of weakenings of henselianity.

\begin{definition}[{\cite[Definition 3.3]{anscombe-jahnke2018henselianity}}]
	Let $n \geq 1$. Say that a valued field $(K,v)$ is \emph{$n_{\leq}$-henselian} if for every $f \in \O_v[X]$ of degree $\leq n$, and $a \in \O_v$, if $v(f(a)) > 0$ and $v(f'(a)) = 0$, then there is $b \in \O_v$ with $f(b) = 0$ and $v(b-a) > 0$.
\end{definition}

\begin{definition}[{\cite[Section 4.2]{engler2005valued}}]
	Let $q$ be a prime. Say that a valued field $(K,v)$ is \emph{$q$-henselian} if $v$ extends uniquely to every Galois extension of $K$ of $q$-power degree.
\end{definition}

\begin{remark}
	A valued field $(K,v)$ is henselian if and only if it is $n_\leq$-henselian for all $n$. Being $n_\leq$-henselian is clearly a first-order property of $(K,v)$; by \cite[Propositions 1.2 and 1.3]{koenigsmann1995p}, the same is true for $q$-henselianity.
\end{remark}

\par Next, we isolate a new notion built along the lines of t-henselianity and divisible-tame type.
\begin{definition}
	We call a field $K$ \emph{t-henselian of defect type} if there is an elementarily equivalent $L \equiv K$ which admits some henselian valuation $v$ such that $(L,v)$ has defect. We call a field $K$ \emph{t-henselian of divisible-defect type} if there is an elementarily equivalent $L \equiv K$ which admits some henselian valuation $v$ such that $(L,v)$ has defect with $vL$ divisible.
\end{definition}

We now replicate a construction that first appeared in \cite{prestel-ziegler1978model}, and was later refined in \cite{fehm-jahnke2015quantifier} and \cite{anscombe-jahnke2018henselianity}. The proofs are almost verbatim the same as in \cite{fehm-jahnke2015quantifier} and \cite{anscombe-jahnke2018henselianity}; we thus give the appropriate references and explain the differences.
\begin{lemma}[{\cite[Lemma~4.8]{anscombe-jahnke2018henselianity}}]\label{building-block}
	Let $p > 0$ be a prime. Let $K$ be a perfect field of characteristic $p$ that contains all roots of unity. Let $n > p$ and let $q$ be a prime with $q > n$. Then, there exists an equicharacteristic valued field $(K',v)$ with
	\begin{itemize}
		\item $K'v = K$, $vK' = \mathbb Q$,
		\item $K'$ is perfect,
		\item $(K',v)$ is not $q$-henselian, but it is $n_{\leq}$-henselian,
		\item $(K',v)$ admits a proper degree $p$ immediate extension.
	\end{itemize}
\end{lemma}
\begin{proof}
	We follow the proof of \cite[Lemma~4.8]{anscombe-jahnke2018henselianity}, but we substitute the generalized power series with the Puiseux series.
	Indeed, inside the Puiseux series $L \coloneqq \bigcup_{n \geq 0} K(\!(t^{\frac{1}{n}})\!)$, endowed with the restriction $v_t$ of the $t$-adic valuation on $K(\!(\mathbb Q)\!)$, consider the subfield $F \coloneqq K(t^{\nu} \mid \nu \in \mathbb Q)$. Consider $\widetilde{F} \coloneqq F^{\mathrm{alg}} \cap L$. Note that since $L$ is henselian, so is $\widetilde{F}$. Then,  arguing as in \cite[Proof of Lemma~4.8]{anscombe-jahnke2018henselianity}, there is a subgroup $G \leq \Gal{F^{\mathrm{sep}}}{F}$ with $G \cong \mathbb Z_q$.
	\par Let $E = \operatorname{Fix}(G) \subseteq F^{\mathrm{sep}}$, and consider $K' \coloneqq E \cap \widetilde{F}$. Then, $(K',v_t)$ has residue field $K$ and value group $\mathbb Q$, and it is perfect. Moreover, the roots of $X^p-X-\frac{1}{t}$ give rise to immediate extensions of degree $p$ over $(K',v_t)$. Now, arguing as in \cite[Claim~4.8.1]{anscombe-jahnke2018henselianity}, $(K',v_t)$ is not $q$-henselian. Similarly, arguing as in \cite[Claim~4.8.2]{anscombe-jahnke2018henselianity}, $(K',v_t)$ is $n_{\leq}$-henselian (note that they argue that $(K',v_t)$ is $(n!^2!)_{\leq}$-henselian, using that in their case $q > n!^2!$).
	\par As for the last point, note that $L$ admits an immediate extension of degree $p$, namely given by any root of the Artin-Schreier polynomial $X^p-X-\frac{1}{t}$. In particular, then, $\widetilde{F}$ also admits an immediate extension of degree $p$, and thus so does $K'$. 
\end{proof}

The following construction follows very closely the proof of {\cite[Proposition~4.13]{anscombe-jahnke2018henselianity}}, using Lemma~\ref{building-block} in place of \cite[Lemma~4.8]{anscombe-jahnke2018henselianity} and diverging only in the last paragraph. We sketch out the construction for the convenience of the reader, but invite them to see \cite{anscombe-jahnke2018henselianity} and \cite{fehm-jahnke2015quantifier} for the full details.

\begin{proposition}[{\cite[Proposition~4.13]{anscombe-jahnke2018henselianity}}]\label{construction-of-t-defect}
	Let $p > 0$ be a prime. There is a non-henselian, t-henselian of divisible-defect type perfect field of characteristic $p$ which is not separably closed.
\end{proposition}
\begin{proof}
	For each $n \geq 0$, set $k_n = n+p+1$ and choose a prime $q_n > k_n$. Let $K_0 = \mathbb F_p^{\mathrm{alg}}$. Then, using Lemma~\ref{building-block}, build a sequence $\{(K_{n+1},\induced{v_n}) \mid n \geq 0\}$ of valued fields such that each $(K_{n+1},\induced{v_{n}})$ has value group $\mathbb Q$, residue field $K_{n}$, is not $q_n$-henselian, but it is ${(k_n)}_{\leq}$-henselian, and it admits a proper degree $p$ immediate extension.
	
	For any $n > m \geq 0$, then, write $v_{n,m} \coloneqq \induced{v_m} \circ \cdots \circ \induced{v_{n-1}}$ and denote by $\O_{n,m}$ the corresponding valuation ring on $K_n$. Denote by $\pi_{n,m} \colon \O_{n,0} \to \O_{m,0}$ the restriction of the residue map $\O_{n,m} \to K_m$.
	Then, the valuation rings $(\O_{n,0})_{n \geq 0}$ form a projective system together with the maps $\pi_{n,m}$.
	The limit $\O$ is again a valuation ring, together with natural projections $\pi_{\infty,n} \colon \O \to \O_{n,0}$. For each $n \geq 0$, consider the localization $\O_{v_n} \coloneqq \O_{\ker(\pi_{\infty,n})}$.
	Each $\O_{v_n}$ is a valuation ring on $K = \operatorname{Frac}(\O) = \bigcup_{n \geq 0} \O_{v_n}$ with residue field $K_n$ and non-trivial divisible value group. 
	Indeed, for each $n \geq 0$, $\O_{v_n} \subseteq \O_{v_{n+1}}$, and $v_{n}$ induces precisely $\induced{v_n}$ on $K_{n+1} = Kv_{n+1}$. 
	\[\begin{tikzcd}
		K && {Kv_{n+1} = K_{n+1}} && {K_n = Kv_n}
		\arrow["{v_{n+1}}"', from=1-1, to=1-3]
		\arrow["{v_n}", curve={height=-18pt}, from=1-1, to=1-5]
		\arrow["{\induced{v_n}}"', from=1-3, to=1-5]
	\end{tikzcd}\]
	We now let $(K^*,v^*)$ be an ultraproduct of the family $(K,v_n)_{n \geq 0}$. Then, $K^*$ is perfect, and $v^*$ is henselian with divisible value group. Moreover, $K$ is not henselian and thus not separably closed.
	
	Now, we diverge from the proof of {\cite[Proposition~4.13]{anscombe-jahnke2018henselianity}} and  argue that $(K^*,v^*)$ admits a degree $p$ immediate extension, in particular a defect extension. It is enough to show that each $(K,v_n)$ admits one such. We know that $(K_{n+1},\induced{v_n})$ admits a degree $p$ immediate extension generated by a root $\alpha$ of an Artin-Schreier polynomial $X^p-X-\xi$, for some $\xi \in K_{n+1}$. Denote by $\induced{u}$ the prolongation of $\induced{v_n}$ to $K_{n+1}(\alpha)$ that makes $(K_{n+1}(\alpha),\induced{u})$ immediate over $(K_{n+1},\induced{v_n})$. For some $z \in K$ with $\res_{v_{n+1}}(z) = \xi$, we consider the Artin-Schreier polynomial $X^p-X-z$. Let $w$ be a prolongation of $v_{n+1}$ to $K^{\mathrm{alg}}$: then there is $a \in K^{\mathrm{alg}}$ with $a^p-a-z = 0$ and $\res_{w}(a) = \alpha$, by henselianity of $w$. Then, if we denote by $w$ again the restriction of $w$ to $K(a)$, we have that $K(a)w = Kw(\alpha) = K_{n+1}(\alpha)$. We now let $u = \induced{u} \circ w$. Then, $K(a)u = K_{n+1}(\alpha)\induced{u} = Kv_n$. Moreover, $v_nK$ is divisible, so $uK(a) = v_nK$.
\end{proof}

\begin{proposition}\label{canonical-has-defect}
	Let $K$ be a non-henselian, t-henselian of defect type, perfect field of characteristic $p$ which is not separably closed. Let $L \equiv K$ be such that, for some henselian valuation $v$, $(L,v)$ has defect. Then $v_L$ has defect.
\end{proposition}
\begin{proof}
	We distinguish two cases. If $v$ is a (possibly non-proper) coarsening of $v_L$, then we are done, since defect goes up in coarsenings. If $v$ is a proper refinement of $v_L$, then $Lv_L$ is separably closed, and since $L$ is perfect, it is in particular algebraically closed. Thus, the valuation $\barv$ induced by $v$ on $Lv_L$ is defectless. But then, since $v = \barv \circ v_L$ has defect if and only at least one of $v_L$ and $\barv$ has defect (\cite[Lemma~2.9]{anscombe2024characterizing}), $v_L$ must have defect.
\end{proof}

\begin{remark}
	The construction in Proposition~\ref{construction-of-t-defect} shows that in equicharacteristic $p$, one cannot eliminate the parameters from Corollary~\ref{cor:define-val-from-indep-defect}. Indeed, assume that, on any perfect henselian valued field $(F,v)$ of characteristic $p>0$ such that $v$ is not defectless, there is a non-trivial $\emptydefinition$-definable henselian coarsening $w$ of $v$. Take a non-henselian, t-henselian of defect type, perfect field $K$ (given by Proposition~\ref{construction-of-t-defect}), and let $L \equiv K$ be such that $(L,v_L)$ has defect (which we can assume exists by Proposition~\ref{canonical-has-defect}). Then, there is a coarsening $w$ of $v_L$ which is $\emptydefinition$-definable. In particular, using the same formula, we can find a non-trivial henselian valuation on $K$, a contradiction.
\end{remark}

\begin{example}\label{example-of-six}
	Let $p>0$ be a prime. Let $K_0$ be a non-henselian, t-henselian of divisible-defect type, perfect field of characteristic $p$ which is not separably closed. Then,
	\[
	K \coloneqq K_0(\!(\mathbb Q)\!)
	\]
	satisfies $\neg 1 \land \neg 2 \land \neg 3 \land \neg 4 \land \neg 5 \land 6$ from our Main Theorem. 
	Indeed, we have \(v_K=v_t\) by Remark~\ref{rem:non-hens-sep-closed-canonical}.
	Now conditions \(\neg 1\), \(\neg 2\), \(\neg 4\), \(6\) are immediate from the construction, \(\neg 5\) follows since $(K_0(\!(\mathbb Q)\!),v_t)$ is tame, and condition \(\neg 3\) follows from Proposition~\ref{prop:t-hens:one-divisible-all-divisible}. It then follows that \(K\) admits a definable non-trivial henselian valuation.
\end{example}

\begin{question}
	Let $K$ be a non-henselian, t-henselian of divisible-tame type field of characteristic $p$ which is not separably closed. 
	Suppose that $L \equiv K$ admits a non-trivial henselian valuation. 
	Must $(L,v_L)$ be defectless? 
	This is the defectless version of Proposition~\ref{prop:t-hens:one-divisible-all-divisible}, 
	and it would yield that if $K_0$ is non-henselian, t-henselian of divisible-tame type, not separably closed, 
	then $K_0(\!(\mathbb Q)\!)$ is an example for $\neg 1 \land \neg 2 \land \neg 3 \land \neg 4 \land \neg 5 \land \neg 6$ (and thus admits no non-trivial definable henselian valuation).
\end{question}

\subsection*{Acknowledgements}
This work was done as part of the PhD projects of the authors: we therefore wish to thank Sylvy Anscombe, Arno Fehm, Martin Hils, and Franziska Jahnke for their \textit{perfect} guidance, highly \textit{valued} encouragement, and dedication in finding and overcoming the \textit{defects} in our arguments. We would like to express special thanks to S. Anscombe for pointing us towards Beth definability in time of need, to F. Jahnke for directing us towards and through these questions, and suggesting the use of independent defect, and to A. Fehm for pointing out how the proof of Proposition \ref{definability} could be simplified and applied in a more general setting. We also wish to thank Blaise Boissonneau, Nicolas Daans, Akash Hossain, and Franz-Viktor Kuhlmann for fruitful discussions.
We thank the anonymous reviewer for the useful comments, that in particular led to the addition of Section 6.

This project began when the third author was a Visiting Doctoral Researcher to the \textit{Cluster of Excellence ``Mathematics Münster''}. Our thanks go to ``Mathematics Münster'' for providing a lively research environment where this project could start.

This research is supported by the Deutsche Forschungsgemeinschaft (DFG, German Research Foundation) under the Excellence Strategy EXC 2044–390685587, \textit{Mathematics Münster: Dynamics–Geometry–Structure}. The first two authors were also supported by the German Research Foundation (DFG) via HI 2004/1-1 (part of the French-German ANR-DFG project GeoMod). The second author was also funded by the DFG project 495759320 -- \textit{Modelltheorie bewerteter Körper mit Endomorphismus}. The third author was also funded by the DFG project 404427454 -- \textit{Definierbarkeit und Entscheidbarkeit in globalen und lokalen Körpern}. 

\bibliographystyle{alpha}

\end{document}